\documentclass{article}

\usepackage{amsmath,amssymb,mathdots,array,booktabs,graphicx,mathrsfs,amsthm}
\usepackage[boxed]{algorithm2e}

\numberwithin{equation}{section}
\def\diag{{\rm diag}}
\newtheorem{Theorem}{Theorem}[section]
\newtheorem{Lemma}{Lemma}[section]

\newtheorem{Remark}{Remark}[section]

\DeclareMathOperator{\rank}{rank}

\newcommand{\bsmat}{\left[\begin{smallmatrix} }
\newcommand{\esmat}{\end{smallmatrix}\right] }

\title{On Some Inverse Eigenvalue Problems of Quadratic Palindromic Systems
\footnotemark[1]}

\author{
Yunfeng Cai
\footnotemark[2]
\and
Jiang Qian
\footnotemark[3]
}
\date{}

\begin{document}
\maketitle
\renewcommand{\thefootnote}{\fnsymbol{footnote}}

\footnotetext[1]{This research was supported in part by NSFC under grant 11301013.}
\footnotetext[2]{LMAM \& School of Mathematical Sciences,
Peking University, Beijing, 100871, P.R. China (yfcai@math.pku.edu.cn)}
\footnotetext[3]{School of Sciences, Beijing University of Posts and Telecommunications, Beijing, 100876, P.R. China (jqian104@126.com)}

\renewcommand{\thefootnote}{\arabic{footnote}}

\begin{abstract}
This paper concerns some inverse eigenvalue problems of the quadratic $\star$-(anti)-palindromic system
$Q(\lambda)=\lambda^2 A_1^{\star}+\lambda A_0 + \epsilon A_1$,
where $\epsilon=\pm 1$, $A_1, A_0 \in \mathbb{C}^{n\times n}$, $A_0^{\star}=\epsilon A_0$, $A_1$ is nonsingular,
and the symbol $\star$ is used as an abbreviation for transpose for real matrices and 
either transpose or conjugate transpose for complex matrices. By using the spectral decomposition of the quadratic $\star$-(anti)-palindromic system, the inverse eigenvalue problems with entire/partial eigenpairs given, and the model updating problems with no-spillover are considered. Some conditions on the solvabilities of these problems are given, and algorithms are proposed to find these solutions. These algorithms are illustrated by some numerical examples.
\end{abstract}

\vskip 2mm



\pagestyle{myheadings}
\thispagestyle{plain}
\markboth{YUNFENG CAI AND JIANG QIAN}{INVERSE EIGENVALUE PROBLEMS OF PALINDROMIC SYSTEMS}

\section{Introduction}\label{introduction}
 
In this paper, we consider the following quadratic $\star$-(anti)-palindromic eigenvalue problem:
\begin{equation}\label{eq:qep}
Q(\lambda)x=0, 
\end{equation}
where 
\begin{align}\label{pal}
Q(\lambda)=\lambda^2 A_1^{\star}+\lambda A_0 + \epsilon A_1,
\end{align}
$\epsilon=\pm 1$, $A_1, A_0 \in \mathbb{C}^{n\times n}$, $A_0^{\star}=\epsilon A_0$, $A_1$ is nonsingular,
and the symbol $\star$ is used as an abbreviation for transpose for real matrices and either transpose or conjugate transpose for complex matrices.
The scalar $\lambda$ and nonzero vector $x$ satisfying \eqref{eq:qep} 
are called an eigenvalue of the quadratic eigenvalue problem (QEP) and the (right) eigenvector of the QEP corresponding to $\lambda$, respectively.
The eigenvalue $\lambda$ together with the corresponding eigenvector $x$, $(\lambda, x)$ is called an eigenpair of the QEP. To be specific, we summarize the names and structures of the palindromic system $Q(\lambda)$ in \eqref{pal} as follows:
\begin{subequations}\label{4pal}
\begin{align}
&\top\mbox{-palindromic} \quad 
&Q(\lambda)&=\lambda^2 A_1^{\top} +\lambda A_0 + A_1,  \quad A_0^{\top}=A_0,\\
&\ast\mbox{-palindromic} \quad 
&Q(\lambda)&=\lambda^2 A_1^{*} +\lambda A_0 + A_1,  \quad A_0^{*}=A_0,\\
&\top\mbox{-anti-palindromic} \quad 
&Q(\lambda)&=\lambda^2 A_1^{\top} +\lambda A_0 - A_1,  \quad A_0^{\top}=-A_0,\\
&\ast\mbox{-anti-palindromic} \quad
&Q(\lambda)&=\lambda^2 A_1^{*} +\lambda A_0 - A_1,  \quad A_0^{*}=-A_0.
\end{align}
\end{subequations}

The phrase ``$\star$-(anti)-palindromic'' is due to the fact that
reversing the order of coefficient matrices of $Q(\lambda)$, 
and followed by taking the (conjugate) transpose and multiplying by $\epsilon$, 
the matrix polynomial becomes nothing but the original one, i.e.,
$Q(\lambda)=\epsilon \lambda^2 Q(1/\lambda)^{\star}$.
Then it follows that eigenvalues of $Q(\lambda)$ occur in pairs $(\lambda,1/\lambda^{\star})$,
or in quadruples $(\lambda,\bar{\lambda},1/{\lambda},1/\bar{\lambda})$ when $A_1$ are $A_0$ are real.
This property is sometimes referred to as ``symplectic spectral symmetry'', 
since symplectic matrices exhibit this behavior.
Moreover, the (algebraic, geometric, and partial) multiplicities of eigenvalues in each pair/quadruple are equal \cite{mackey2006structured}.

The $\star$-(anti)-palindromic QEPs arise in the analysis and numerical solution of high order systems of ordinary and partial differential equations, 
and enjoy a variety of applications.
For example, a $\top$-palindromic QEP is raised in the study of rail traffic noise caused by high speed trains \cite{hilliges2004numerische,hilliges2004solution},
and of the behavior of periodic surface of acoustic wave filters \cite{zaglmayr2002eigenvalue};
an $*$-palindromic QEP arises in the computation of Crawford number by bisection and level set methods \cite{higham2002detecting}.
For more examples, we refer readers to \cite{mackey2006structured,chu2010palindromic,huang2011palindromic} and references therein.


The QEP is usually solved via a two-stage procedure: first, transform the QEP into a general eigenvalue problem (GEP) via {\em linearization}; second, apply certain numerical methods for the GEP, such as the QZ method \cite{van1996matrix, moler1973algorithm}, etc. However, the symplectic spectrum symmetry is in general destroyed, producing large numerical errors \cite{ipsen2004accurate,chu2008vibration,huang2009structure}.
Great efforts have been made to the development of numerical methods which preserve the symplectic spectrum symmetry. In \cite{mackey2006structured, mackey2006vector}, the $\star$-(anti)-palindromic polynomial eigenvalue problem is linearized into the form $\lambda Z^{\star} + Z$, which naturally preserves the symplectic spectrum symmetry and enables the developments of structure preserving numerical methods, for example, the QR-like method \cite{schroder2007qr}, the hybrid method computing the anti-triangular Schur form \cite{mackey2009numerical} and the URV decomposition based structured method \cite{schroder2007urv}. Based on a symplectic linearization, a structure preserving doubling algorithm is developed in \cite{li2011palindromic} for linear palindromic eigenvalue problem. In \cite{chu2008vibration,huang2009structure,chu2012structured}, the structure preserving doubling algorithms are discussed for the $\star$-(anti)-palindromic QEPs.

The inverse eigenvalue problem (IEP), on the other hand, concerns determining the coefficient matrices
with partial or entire eigenvalues (and eigenvectors) prescribed.
The IEP is of great importance and challenging, many efforts have been devoted to the IEP of quadratic systems, especially quadratic symmetric systems.
There is quite a long list of studies on this subject, 
see for example,  \cite{friswell1995finite, lancaster2005isospectral,lancaster2005inverse, kuo2006new, lancaster2007inverse, cai2009solutions, cai2009quadratic,chu2009spectral,cai2010new, jia2011real,mao2014structure} and references therein. The IEP of $\star$-(anti)-palindromic systems with entire eigenvalues given is considered in \cite{Al-Ammari2011}, which is based on the spectral decomposition of $\star$-(anti)-palindromic systems \cite{lancaster2007canonical, Al-Ammari2011}.

In this paper, we consider some other types of IEP of quadratic $\star$-(anti)-palindromic systems \eqref{pal},
namely, the IEP of quadratic $\star$-(anti)-palindromic systems with $k$ prescribed eigenpairs (IEP-QP($k$)), where $1\le k \le 2n$, and the model updating problems with no-spillover of quadratic $\star$-(anti)-palindromic systems (MUP-QP).

The IEP of quadratic symmetric systems with $k$ prescribed eigenpairs is discussed in \cite{chu2004inverse,kuo2006solutions} for $1\le k \le n$,
under the assumption that the eigenvector matrix, formed by the prescribed eigenvectors, is of full column rank.
For the case when $n<k\le 2n$, a general solution is given in \cite{cai2009solutions} assuming that the eigenvector matrix is of full row rank.
Special attention is paid to the IEP of quadratic symmetric systems with entire eigenpairs given in \cite{cai2009quadratic}, 
where under some proper assumptions,
 it is shown that when the solution is not unique up to a scalar, 
the coefficient matrices of the resulting system can be jointly block diagonalized via a congruence transformation. 

In this paper, the IEP of quadratic $\star$-(anti)-palindromic systems with $k$ prescribed eigenpairs (IEP-QP($k$)) is considered in section~\ref{app}. We first propose a necessary and sufficient condition on the existence of regular solutions to the IEP-QP($2n$). It is also shown that under certain conditions,  
the coefficient matrices of the solutions to the IEP-QP($2n$) can be jointly block diagonalized.
We then propose an algorithm to solve the IEP-QP$(k)$ for general $k$,
based on the spectral decomposition of the  $\star$-(anti)-palindromic QEP.
Different from the approaches on solving the IEP for symmetric systems in literature, we give solutions to the IEP-QP($k$) uniformly, 
without distinguishing the cases of $ k\le n$ and $k>n$, and we do not need to assume that the eigenvector matrix is of full column or row rank.

The model updating problem with no-spillover (MUP) is to replace some unwanted eigenvalues of the original system by some desired ones, meanwhile keeping the remaining eigenpairs unchanged (the so called no-spillover phenomenon). There exist some references on the MUP for quadratic symmetric systems. In \cite{carvalho2006symmetry}, the unwanted eigenvalues is replaced one by one or one complex conjugate pair by pair. However, this approach
may suffer from breakdown before all desired eigenvalues are updated.
Later in \cite{chu2009spectral}, the unwanted eigenvalues are replaced simultaneously.
Parametric solutions are given in \cite{cai2010new},
in which the parameters can be further exploited to optimize the updated system in some sense. 
In these approaches for symmetric systems, the number of eigenvalues to be replaced is restricted to be no more than $n$, and the range space spanned by those desired eigenvectors of the updated system is assumed to be in that of those eigenvectors to be replaced of the original system. Otherwise, they fail to output any solution.

In this paper, the MUP of quadratic $\star$-(anti)-palindromic systems (MUP-QP) is considered in section~\ref{eep}.
Parametric solutions are given for two cases:
 the eigenvectors of the updated system corresponding to the updated eigenvalues are prescribed, and not prescribed. These solutions are given without the restrictions that the number of unwanted eigenvalues must be  no more than $n$, and and the range space spanned by those desired eigenvectors of the updated system must be in that of those eigenvectors to be replaced of the original system.
%

The rest of this paper is organized as follows.
In section~\ref{sd}, we present the spectral decomposition of the $\star$-(anti)-palindromic QEP, expressing the coefficient matrices in terms of the standard pair, and a parameter matrix, whose structure is characterized in
section~\ref{ss}. The spectral decomposition is used to solve some QIEPs:
the IEP-QP$(k)$ and the MUP-QP in sections~\ref{app} and \ref{eep}, respectively. Some numerical examples are presented in section~\ref{example} to illustrate the performance of the algorithms proposed in sections~\ref{app} and \ref{eep}. Some concluding remarks are finally drawn in section~\ref{conclusion}.

\section{Preliminaries}\label{sd}
In this section, we present some preliminaries results which will be used to solve the inverse eigenvalue problems of quadratic palindromic systems. 

\subsection{Spectral decomposition}
We start with the spectral decomposition theory of the $\star$-(anti)-palindromic QEP, expressing the coefficient matrices $A_1$ and $A_0$ in terms of a standard pair $(X, T)$ together with a parameter matrix $S$. Here a matrix pair $(X,T)\in\mathbb{C}^{n\times 2n}\times \mathbb{C}^{2n\times 2n}$ is called a {\em standard pair} of  $Q(\lambda)$ \eqref{pal} if the matrix $W=W(X,T):=\bsmat X\\  -XT^{-1} \esmat$ is nonsingular and 
\begin{equation}\label{eq:qepmatright}
A_1^{\star}XT^2+A_0XT+\epsilon A_1X=0.
\end{equation}
If all eigenvalues of $Q(\lambda)$ \eqref{pal} are semi-simple, we may collect all eigenvalues in a diagonal matrix $\Lambda$ and corresponding $2n$ eigenvectors in $U$, and then $(U,\Lambda)$ serves as a standard pair of $Q(\lambda)$ \eqref{pal}. However if some eigenvalue of $Q(\lambda)$ \eqref{pal} is defective, $Q(\lambda)$ does not have a complete set of $2n$ eigenvectors. In this case, the standard pair is still well defined, and similar as Lemma 2.2 in \cite{chu2009spectral}, the eigenvalues and eigenvectors of $Q(\lambda)$ can be completed attainable by the standard pair.

For a matrix pair $(X,T)\in\mathbb{C}^{n\times k}\times \mathbb{C}^{k\times k}$,
define
\begin{align}
\mathbb{S}_T&=\{S\in\mathbb{C}^{ k \times k} \; \big| \; S^{\star} = -\epsilon S, S=T S T^{\star} \},\\
\mathbb{S}_{(X,T)}&=\{S\in\mathbb{S}_T \; \big| \; X S X^{\star}=0 \}.
\end{align}
Then the spectral decomposition of the $\star$-(anti)-palindromic QEP is characterized in the following theorem.

\begin{Theorem}\label{thm:sp}
Let $(X,T)$ be a standard pair of a regular $\star$-(anti)-palindromic system $Q(\lambda)$.
Denote
\begin{equation}\label{s}
S=(W^{\star} L J_{\epsilon} L^{\star} W)^{-1},
\end{equation}
where 
\begin{align}\label{WLJ}
W=\begin{bmatrix} X\\ -XT^{-1}\end{bmatrix},\qquad 
L=\begin{bmatrix}0 & I \\ A_1^{\star} & 0\end{bmatrix},\qquad
J_{\epsilon}=\begin{bmatrix}0 & I_n \\ - \epsilon I_n & 0\end{bmatrix}.
\end{align}
Then $S\in\mathbb{S}_{(X,T)}$, and 
the coefficient matrices $A_1$, $A_0$
can be represented in terms of $X$, $T$ and $S$ as:
\begin{equation}\label{a1a0}
A_1=\epsilon (XT^{-1}SX^{\star})^{-1},\qquad A_0=- A_1XT^{-2}SX^{\star}A_1.
\end{equation}
\end{Theorem}

\begin{proof}
Since $(X,T)$ is a standard pair of the regular $\star$-(anti)-palindromic system $Q(\lambda)$, both $W$ and $L$ defined in \eqref{WLJ} are nonsingular. So $S$ in \eqref{s} is well defined, and obviously satisfies $S^{\star} = -\epsilon S$, since $J_{\epsilon}^{\star} = -\epsilon J_{\epsilon}$.

It follows from \eqref{eq:qepmatright} that
\begin{equation}\label{eq:qepmatleft}
T^{-2\star} X^{\star} A_1^{\star} +T^{-\star} X^{\star} A_0 +X^{\star} \epsilon A_1 =0,
\end{equation}
which is equivalent to
\begin{align}\label{eq:vmjvl}
 W^{\star} M=T^{-\star} W^{\star} L,
\end{align}
where $M=\begin{bmatrix}
\epsilon A_1 & 0\\ -A_0 & -I
\end{bmatrix}$, $W$ and $L$ are defined in \eqref{WLJ}. Direct calculations show that these $M,L$ and $J_{\epsilon}$ satisfy
\begin{align}\label{mjm}
MJ_{\epsilon}M^{\star}
=\begin{bmatrix}0 & -\epsilon A_1 \\ A_1^{\star} & 0\end{bmatrix}
=LJ_{\epsilon}L^{\star}.
\end{align}
Using \eqref{s}, \eqref{mjm} and \eqref{eq:vmjvl}, we have
\begin{align*}
S^{-1}=W^{\star} LJ_{\epsilon}L^{\star} W
=W^{\star} MJ_{\epsilon}M^{\star} W
=T^{-\star}W^{\star} L J_{\epsilon} L^{\star} W T^{-1}
=T^{-\star} S^{-1} T^{-1},
\end{align*}
from which we can see that $S=T S T^{\star}$. Hence we have $S\in\mathbb{S}_T$.

From \eqref{s} and the definition of $W$ in \eqref{WLJ},  by calculation we have
\begin{align*}
\begin{bmatrix}
XSX^{\star} & -XST^{-\star}X^{\star}\\
-XT^{-1}SX^{\star} & XT^{-1}ST^{-\star}X^{\star}
\end{bmatrix}=
WSW^{\star}=(LJ_{\epsilon}L^{\star})^{-1}
=\begin{bmatrix}0 &  A_1^{-\star}  \\ -\epsilon A_1^{-1} & 0 \end{bmatrix}.
\end{align*}
The above equation gives $XSX^{\star}=0$, which implies $S\in\mathbb{S}_{(X,T)}$, 
and the formula for $A_1$ as in \eqref{a1a0}. The formula for $A_0$ in \eqref{a1a0} then follows readily by postmultiplying $T^{-2}SX^{\star}$ on \eqref{eq:qepmatright}.
\end{proof}

Theorem~\ref{thm:sp} is motivated by Chu and Xu \cite{chu2009spectral} for symmetric systems, which can also be obtained via the GLR theory \cite{gohberg2009matrix}. It also follows from cases (4) and (5) with $m=2$ of Theorem~3.5 in \cite{Al-Ammari2012}: the parameter matrix $S$ in Theorem~\ref{thm:sp} is just $-\mathcal{S}v_{\mathcal{S}}(\mathcal{T})$ and the equations $S^{\star} = -\epsilon S, S=T S T^{\star}$ in the definition of $\mathbb{S}_T$ follows from the equation $\mathcal{T}\mathcal{S}^{\star}=-\alpha\epsilon \mathcal{S}$ in Theorem 3.5 of \cite{Al-Ammari2012}.

\subsection{The structure of $S\in\mathbb{S}_T$}\label{ss}
We now take a deep look at the structure of the parameter matrix $S\in\mathbb{S}_T$. Recall that the eigenvalues of  QEP \eqref{eq:qep} occur in pairs $(\lambda,\frac{1}{\lambda^{\star}})$,
with the same algebraic multiplicities and same partial multiplicities.
Let the distinct eigenvalues of $Q(\lambda)$ be
\begin{align}
\lambda_1,\lambda_2,\dots,\lambda_{2\ell-1},\lambda_{2\ell}, \lambda_{2\ell+1},\dots,\lambda_{t},
\end{align}
where for $i=1,\dots,\ell$,$\lambda_{2i}=\frac{1}{\lambda_{2i-1}^{\star}}$, $|\lambda_{2i-1}| \le 1$,
and for $i=2\ell+1,\dots, t$, $\lambda_i=\pm 1$  for $\top$-(anti)-palindromic system 
or $|\lambda_i|=1$ for $*$-(anti)-palindromic system.
For $i=1,\dots, t$,
assume that $\lambda_i$ has algebraic multiplicity  $n_i$,  geometry multiplicity $m_i$, 
and partial multiplicities $n_{i1}\ge n_{i2}\ge \dots \ge n_{im_i}$. 
Then we can denote the Jordan block associated with $\lambda_i$ by 
\begin{equation}
J_i=\lambda_i I_{n_i}+N_i, \quad N_i=\diag(N_{i1},\dots,N_{im_i}),
\end{equation}
where $N_{ij}\in\mathbb{R}^{n_{ij}\times n_{ij}}$ is a nilpotent matrix with $n_{ij}-1$ ones along its superdiagonal for each $j=1,\dots, m_i$.
Therefore, $\sum_{i=1}^{t}n_i=\sum_{i=1}^t\sum_{j=1}^{m_i}n_{ij}=2n$, 
and for $i=1,\dots, \ell$,  $j=1,\dots, m_i$,  $n_{2i-1}=n_{2i}$, $n_{2i-1,j}=n_{2i,j}$.
Let 
\begin{align}
J=\diag(J_1, J_2,\dots, J_t).
\end{align}
Hereafter $J$ of the above form will be referred to as the {\em palindromic Jordan canonical form} (PJCF).

Let $(X, J)$ be a stand pair of QEP \eqref{eq:qep} with $J$ of the PJCF.
Direct calculations show that $S\in\mathbb{S}_J$ if and only if $S$ is of the form
\begin{equation}\label{eq:s}
S=\diag\left(\begin{bmatrix} 0 & S_1\\ -\epsilon S_1^{\star} & 0\end{bmatrix},\dots, \begin{bmatrix} 0 & S_{\ell}\\ -\epsilon S_{\ell}^{\star} & 0\end{bmatrix}, S_{2\ell+1},\dots, S_t\right),
\end{equation}
where for $i=1,\dots, t$, $S_i$ satisfies 
\[
(\lambda_i I+N_i)S_i(\frac{1}{\lambda_i^{\star}}I+N_i)^{\star}=S_i,
\]
and the matrices $S_i$ display the sign characteristic of $Q(\lambda)$ \cite{Al-Ammari2012}.
Partition $S_i$ as $S_i=[S^{(i)}_{jk}]$, where $S^{(i)}_{jk}\in\mathbb{C}^{n_{ij}\times n_{ik}}$. It then holds that
\begin{equation}\label{eq:jsj}
(\lambda_i I+N_{ij} )S^{(i)}_{jk} (\frac{1}{\lambda_i^{\star}}I+N_{ik})^{\star}=S^{(i)}_{jk}.
\end{equation}
Notice that $(\frac{1}{\lambda_i^{\star}}I+N_{ik})^{-\star}$ is similar to $\lambda_i I+N_{ik}^{\top} $, i.e., 
there exists a nonsingular matrix $P_{ik}$ such that $(\frac{1}{\lambda_i^{\star}}I+N_{ik})^{-\star}=P_{ik}^{-1}(\lambda_i I+N_{ik}^{\top})P_{ik}$.
To be specific, $P_{ik}$ can be given by 
\begin{equation}\label{p}
P_{ik}=\diag(\lambda_i^{n_{ik}-1},\lambda_i^{n_{ik}-2},\dots,1)L_{ik}
\diag(1,-1/\lambda_i,\dots, 1/(-\lambda_i)^{n_{ik}-1}),
\end{equation}
where $L_{ik}$ is a lower triangular  {\em Pascal-like matrix} of order $n_{ik}$.
Here a lower triangular  Pascal-like matrix $L$ of order $m$ is defined as $L=[l_{ij}]$,
where $l_{ij}=0$ for $i<j$, and $l_{ij}=C({m-j}, {m-i})$ for $i\ge j$. The symbol $C(\cdot,\cdot)$ stands for the binomial coefficient. Now that \eqref{eq:jsj} becomes 
\[
N_{ij} S^{(i)}_{jk} P_{ik}^{-1}=  S^{(i)}_{jk} P_{ik}^{-1} N_{ik}^{\top}.
\]
Using the results in \cite{chu2009spectral}, we know that each $S^{(i)}_{jk} P_{ik}^{-1}$ is an upper triangular Hankel matrix. 
Then it follows that $S^{(i)}_{jk}=P_{ik}H_{ik}$, where $P_{ik}$ is a  lower triangular scaled  rotated {Pascal matrix}, and $H_{ik}$ is an upper triangular Hankel matrix. 
In another word, $S^{(i)}_{jk}$ is the product of a scaled rotated Pascal matrix and a Hankel matrix.
Hereafter, we will refer matrices with such structure as {\em Pascal-Hankel} matrices.
 
So, overall speaking, the parameter matrix $S$ is in a block diagonal form \eqref{eq:s},
and each $S_i$ is a block Pascal-Hankel matrix.
It is worth mentioning here that, in $S$, the Hankel matrices $H_{ik}$ are free parameters,
while the matrices $P_{ik}$ are all constant.

When all eigenvalues are semi-simple, then $S$ is still in the form \eqref{eq:s}, 
and each $S_i$ has only one block, which is a Pascal-Hankel matrix.
Furthermore, when all eigenvalues of $Q(\lambda)$ are simple, let
\begin{equation}\label{simplel}
\Lambda=\diag(\lambda_1,\lambda_2,\dots,\lambda_{2\ell},\lambda_{2\ell+1},\dots,\lambda_{2n}),
\end{equation}
where  for $i=1,2,\dots,\ell$, $\lambda_{2i}=\frac{1}{\lambda_{2i-1}^{\star}}$,
and for $i=2\ell+1,\dots, 2n$, $\lambda_i=\pm 1$ for $\top$-(anti)-palindromic system or $|\lambda_i|=1$ for $*$-(anti)-palindromic system,
then $S\in\mathbb{S}_{\Lambda}$ if and only if
\begin{equation}\label{simples}
S=\diag\left(\begin{bmatrix}0 & s_1\\ -\epsilon s_1^{\star} & 0\end{bmatrix},\dots,
\begin{bmatrix}0 & s_{\ell}\\ -\epsilon s_{\ell}^{\star} & 0\end{bmatrix},
s_{2\ell+1},\dots, s_{2n}\right),
\end{equation}
where for $i=1,2,\dots,\ell$, $s_i\in\mathbb{C}$,
and for $i=2\ell+1,\dots, 2n$, $s_i^{\star}=-\epsilon s_i$.

\begin{Remark}{\rm
For  $*$-palindromic and $*$-anti-palindromic systems, $s_i^{\star}=-\epsilon s_i$ implies that
$s_i$ is pure imaginary and real, respectively.
For $\top$-anti-palindromic system, $s_i^{\star}=-\epsilon s_i$ holds for any $s_i\in\mathbb{C}$.
However, for $\top$-palindromic system, $s_i^{\star}=-\epsilon s_i$ implies $s_i=0$,
which seems to contradict with the fact that the parameter matrix $S$ should be nonsingular for a regular $\top$-palindromic system.
But, in fact, $1(-1)$ can not be a simple eigenvalue of a regular $\top$-palindromic system.
This is because when $1(-1)$ is an simple eigenvalue, then $-1(1)$ is also an eigenvalue with an odd algebraic multiplicity.
Therefore, the product of all eigenvalues $\Pi_{i=1}^{2n}\lambda_i$ equals to $-1$, which contradicts with 
$\Pi_{i=1}^{2n}\lambda_i=\det(A_1)/\det(A_1^{\top})=1$.
}\end{Remark}

\subsection{A special standard pair}
Let $(X,T)$ be a standard pair of $Q(\lambda)$, where $T$ is not necessarily of the PJCF, and $S$ be the corresponding parameter matrix. Direct calculation shows that $(XY,Y^{-1}TY)$ is also a standard pair of $Q(\lambda)$ for any nonsingular matrix $Y$, and the corresponding parameter matrix becomes 
$Y^{-1}SY^{-\star}$. 
We wish to obtain a special standard pair by introducing appropriate nonsingular matrix $Y$, such that the corresponding parameter matrix is as simple as possible.

We will need the following results. For any matrix $B\in\mathbb{C}^{n\times n}$ satisfying $B^{\star}=-\epsilon B$ and $\rank(B)=t$,
it can be factorized as
\begin{align}\label{bzgz}
B=Z\Gamma Z^{\star},
\end{align}
where $Z$ is unitary,  and
\begin{align}\label{Gamma}
\Gamma=
\begin{cases}
\diag(\gamma_1,\gamma_2,\dots,\gamma_t,0,\dots, 0),\quad  0\ne \imath\gamma_i\in\mathbb{R}, \quad &\mbox{if } \star=*, \epsilon=1; 
\\
\diag(\gamma_1,\gamma_2,\dots,\gamma_t,0,\dots, 0),\quad 0\ne \gamma_i\in\mathbb{R}, \quad &\mbox{if } \star=*, \epsilon=-1;\\
\diag(\bsmat 0 & \gamma_1\\ -\gamma_1 & 0\esmat, \dots, \bsmat 0& \gamma_{t/2}\\ -\gamma_{t/2} & 0\esmat,0,\dots, 0),\quad \gamma_i>0, \quad &\mbox{if } \star=\top, \epsilon=1,
\\
\diag(\gamma_1,\gamma_2,\dots,\gamma_t,0,\dots, 0),\quad  \gamma_i>0, \quad &\mbox{if } \star=\top, \epsilon=-1.
\end{cases}
\end{align}
The first two cases ($\star=*$, $\epsilon=\pm1$) are actually the Schur decompositions of $B$;
the last two cases ($\star=\top$, $\epsilon=\pm1$) are  Hua's decompositions \cite[Theorem 7.5]{hua1944basis}. 
Using permutation and diagonal scaling, we can further factorize $\Gamma$ as 
\begin{equation}\label{Gammafac}
\Gamma= PD \Delta DP^{\top},
\end{equation}
where $P$ is a permutation matrix, $D>0$ is diagonal, and
\begin{equation}\label{Delta}
\Delta=
\left\{\begin{array}{lll}
\Delta_{p,q},   & p+q=t, &\mbox{if } \star=* ;\\
\Delta_t, & & \mbox{if } \star=\top.
\end{array}\right.
\end{equation}
Here
\begin{subequations}\label{Delta2}
\begin{align}
\Delta_{p,q}&=
\left\{\begin{array}{ll}
\diag(\imath I_q,-\imath I_p,0),  &\mbox{if }  \epsilon=1,
\\
\diag(I_p,-I_q ,0), &\mbox{if }  \epsilon=-1,
\end{array}\right.\label{Deltapq}
\\
\Delta_t&=
\left\{\begin{array}{ll}
\diag(\bsmat 0 & I_{t/2}\\ -I_{t/2} & 0\esmat, 0), & \mbox{if } \epsilon=1,
\\
\diag(I_t, 0), &  \mbox{if }  \epsilon=-1,
\end{array}\right.\label{Deltat}
\end{align}
\end{subequations}
where $p$, $q$ are the positive and negative inertia indices of $\sqrt{-\epsilon}B$, respectively.
Then let $Y=ZPD$, \eqref{bzgz} can be rewritten as 
\begin{align}\label{ydy}
B=Y\Delta Y^{\star}.
\end{align}
Hereafter, we call the factorization $B=Y\Delta Y^{\star}$ 
as the {\em $\star$-factorization} of $B$, 
the matrices $Y$ and $\Delta$ as the $Y$ and $\Delta$ {\em factors} of $B$, respectively. 

Now let $(X,T)$ be a standard pair of $Q(\lambda)$, and $S$ be the corresponding parameter matrix.
Then $S$ is nonsingular and satisfies $S^{\star}=-\epsilon S$.
Therefore, the $\star$-factorization of $S$ is of the form
\begin{equation}\label{facS}
S=\widehat{Y}\widehat{\Delta} \widehat{Y}^{\star},
\end{equation}
where $\widehat{Y}$ is nonsingular, and
\begin{equation}\label{Delta2n}
\widehat{\Delta}=
\left\{\begin{array}{ll}
\Delta_{n,n},   &\mbox{if } \star=* ;\\
\Delta_{2n}, & \mbox{if } \star=\top.
\end{array}\right.
\end{equation}
Here $\Delta_{n,n}$ and $\Delta_{2n}$ are defined in \eqref{Deltapq} and \eqref{Deltat}, respectively.
And using \eqref{s} we can show $p=q=n$ in $\Delta_{n,n}$.

Let $(\widehat{X}, \widehat{T})=(X \widehat{Y}, \widehat{Y}^{-1}T \widehat{Y})$, 
then $(\widehat{X}, \widehat{T})$ is also a standard pair of $Q(\lambda)$,
and the corresponding parameter matrix is $\widehat{\Delta}$.
Using $\widehat{\Delta}\in\mathbb{S}_{\widehat{T}}$, we have
$\widehat{\Delta}=\widehat{T}\widehat{\Delta}\widehat{T}^{\star}$.
Then $\widehat{T}$ is $U$-symplectic for $\star=*$, 
that is,  $U^* \widehat{T}U$ is conjugate symplectic, where $U_{2n}=\frac{1}{\sqrt{2}}\bsmat I_n & I_n \\ \imath I_n & -\imath I_n\esmat$;
symplectic for $\star=\top$, $\epsilon=1$;
complex orthogonal for $\star=\top$, $\epsilon =-1$.
Notice here that the eigenvalues of $U$-symplectic matrix and symplectic/complex orthogonal matrix
 appear in pairs $(\mu,\frac{1}{\bar{\mu}})$ and $(\mu,\frac{1}{\mu})$, respectively.

\section{Solving IEP-QP($k$)}\label{app}
In this section, we consider the inverse eigenvalue problem of quadratic $\star$-(anti)-palindromic systems with $k$ prescribed eigenpairs (IEP-QP($k$)), which can be stated as follows.

{\em \textbf{IEP-QP$(k)$}: Given $(X_1,T_1)\in\mathbb{C}^{n\times k}\times \mathbb{C}^{k\times k}$ with $\bsmat X_1\\-X_1T_1^{-1}\esmat$ of full column rank ($1\le k \le 2n$), construct a  $\star$-(anti)-palindromic system $Q(\lambda)$ such that 
\begin{equation}\label{eq:iep}
A_1^{\star}X_1T_1^2+A_0 X_1T_1 +\epsilon A_1 X_1=0.
\end{equation}
}

The solution $Q(\lambda)$ is referred to as a regular solution if the leading coefficient matrix is nonsingular.
Note the IEP-QP($k$) considered here is slightly different from the problem considered in \cite{Al-Ammari2011}, where only eigenvalues are given while eigenvectors are not prescribed.
%
%

\subsection{Case $k=2n$} 
We first consider the special case when $k=2n$, that is, all eigenpairs are given. The following theorem presents  a necessary and sufficient condition on the existence of a regular solution to the IEP-QP$(2n)$.

\begin{Theorem}\label{thm:iep2n}
Let $(X,T)\in\mathbb{C}^{n\times 2n}\times \mathbb{C}^{2n\times 2n}$ and 
$W=\bsmat X\\ -XT^{-1} \esmat$ be nonsingular.
Then there exists a regular solution $Q(\lambda)$ to the IEP-QP$(2n)$
if and only if there exists a nonsingular $S\in\mathbb{S}_{(X,T)}$.
In such case, the coefficient matrices $A_1$, $A_0$ of $Q(\lambda)$ are given by \eqref{a1a0}.
\end{Theorem}
\begin{proof}
The necessity can be directly followed from Theorem~\ref{thm:sp}. 
Next, we only show the sufficiency.
Let $S$ be a nonsingular matrix and $S\in\mathbb{S}_{(X,T)}$.
Define $A_1$, $A_0$ as in \eqref{a1a0}, 
 direct calculation leads to
\begin{align*}
&(A_1^{\star}XT^2+A_0XT+\epsilon A_1X)T^{-2}SW^{\star}\\
=&(A_1^{\star}X+A_0XT^{-1}+\epsilon A_1XT^{-2})S[X^{\star} \; |\; T^{-\star}X^{\star} ]\\
=&[0+A_0(\epsilon A_1^{-1})+\epsilon A_1(- A_1^{-1}A_0A_1^{-1}) \; | \; 
A_1^{\star}(- \epsilon^2 A_1^{-\star})+0+\epsilon A_1(\epsilon A_1^{-1})]=0.
\end{align*}
Hence $A_1^{\star}XT^2+A_0XT+\epsilon A_1X=0$, which completes the proof of sufficiency.
\end{proof}

For each nonsingular $S\in\mathbb{S}_{(X,T)}$, if it exists, \eqref{a1a0} gives a regular solution $Q(\lambda)$ to the IEP-QP$(2n)$. If such $S$ is unique up to scalars, these solutions $Q(\lambda)$ to the IEP-QP$(2n)$ are also unique up to scalars. However, if such $S$ is not unique up to scalars, we will show that under some mild conditions, the coefficient matrices of $Q(\lambda)$ can be jointly block diagonalized. We will need the following two lemmas.


\begin{Lemma}\label{lem:jbd1}
Given $(X,T)\in\mathbb{C}^{n\times 2n}\times \mathbb{C}^{2n\times 2n}$ with $\bsmat X\\ -XT^{-1}\esmat$ nonsingular. 
Suppose that there exist two nonsingular matrices $S, \widetilde{S}\in\mathbb{S}_{(X,T)}$.
Then the eigenvalues of $\widetilde{S}S^{-1}$ have the pairing $(\mu,\mu^{\star})$.
If $\mu\ne \mu^{\star}$, they have the same algebraic multiplicity;
otherwise, the multiplicity of $\mu$ must be even.
\end{Lemma}

\begin{proof}
Define $Q(\lambda)$ as in \eqref{pal}, where $A_1$ and $A_0$ are given by \eqref{a1a0} in terms of $X$, $T$ and $S$. Likewise, define $\widetilde{Q}(\lambda)$ with the coefficient matrices given by \eqref{a1a0} in terms of $X$, $T$ and $\widetilde{S}$.
By Theorem~\ref{thm:iep2n}, we know that $(X,T)$ is a stand pair of $Q(\lambda)$ and also $\widetilde{Q}(\lambda)$.
Now let $W=\bsmat X \\ -XT^{-1}\esmat$, we have 
\begin{align*}
S^{-1}=W^{\star}\bsmat 0 & -\epsilon A_1 \\ A_1^{\star} & 0\esmat W,\quad
\widetilde{S}^{-1}=W^{\star}\bsmat 0 & -\epsilon \widetilde{A}_1 \\ \widetilde{A}_1^{\star} & 0\esmat W.
\end{align*}
Then it follows that
\begin{align}\label{sswaw}
\widetilde{S}S^{-1}=W^{-1} \diag(\widetilde{A}_1^{-\star}A_1^{\star}, \widetilde{A}_1^{-1}A_1) W.
\end{align}
Noticing that 
\begin{align*}
\widetilde{A}_1^{-1} A_1 x=\mu x 
\Longleftrightarrow (x^{\star}A_1^{\star}) (\widetilde{A}_1^{-\star} A_1^{\star})= \mu^{\star} (x^{\star}A_1^{\star}),
\end{align*}
the conclusion follows immediately.
\qquad \end{proof}

\begin{Lemma}\label{lem:jbd2}
Let $(X, J)$ be a stand pair of $Q(\lambda)$ with $J$ of the PJCF.
Assume the geometry multiplicity of each distinct eigenvalue is one.
Then for any nonsingular $S,\widetilde{S}\in\mathbb{S}_{J}$ of the form
\begin{align*}
S&=\diag\left(\bsmat 0 & S_1\\ -\epsilon S_1^{\star} & 0\esmat,\dots, \bsmat 0 & S_{\ell}\\ -\epsilon S_{\ell}^{\star} & 0\esmat, S_{2\ell+1},\dots, S_t\right),\\
\widetilde{S}&=\diag\left(\bsmat 0 & \widetilde{S}_1\\ -\epsilon \widetilde{S}_1^{\star} & 0\esmat,\dots, 
\bsmat 0 & \widetilde{S}_{\ell}\\ -\epsilon \widetilde{S}_{\ell}^{\star} & 0\esmat, \widetilde{S}_{2\ell+1},\dots, \widetilde{S}_t\right),
\end{align*}
$\widetilde{S}S^{-1}$ is of the form
\begin{align}\label{ssform}
\widetilde{S}S^{-1}=\diag(\widetilde{S}_1^{-\star}S_1^{\star},\widetilde{S}_1^{-1}S_1,
\dots,\widetilde{S}_{\ell}^{-\star}S_{\ell}^{\star},\widetilde{S}_{\ell}^{-1}S_{\ell},
\widetilde{S}_{\ell+1}^{-1}S_{\ell+1},\dots, \widetilde{S}_{t}^{-1}S_{t}),
\end{align}
and for each $i=1,2,\dots,t$, $\widetilde{S}_i^{-1}S_i$ is a lower triangular Toeplitz matrix.
\end{Lemma}

\begin{proof}
Direct calculation shows that $\widetilde{S}S^{-1}$ is in the block diagonal form \eqref{ssform}.
Next we only need to show each $\widetilde{S}_i^{-1}S_i$ is a lower triangular Toeplitz matrix.

Using $S,\widetilde{S}\in\mathbb{S}_{J}$, we have $S_i, \widetilde{S}_i\in\mathbb{S}_{J_i}$
and hence, $S_i=J_i S_i J_i^{\star}$ and $\widetilde{S}_i=J_i \widetilde{S}_i J_i^{\star}$.
Then it follows that $J_i^{\star}\widetilde{S}_i^{-1}S_i=  \widetilde{S}_i^{-1}S_i J_i^{\star}$.
Using $J_i=\lambda_i I_{n_i} +N_i$, we get 
$N_i^{\top}\widetilde{S}_i^{-1}S_i=  \widetilde{S}_i^{-1}S_i N_i^{\top}$.  
Comparing the elements on both sides of the above equality, 
we know that $\widetilde{S}_i^{-1}S_i$ is a lower triangular Toeplitz matrix.
\qquad\end{proof}

Using Lemma~\ref{lem:jbd1},
we can assume that the eigenvalues of $\widetilde{S}^{-1}S$, counting multiplicities,
 are
\begin{align}\label{eigss}
\underbrace{\mu_1,\dots,\mu_1}_{n_1},
\underbrace{\mu_1^{\star},\dots,\mu_1^{\star}}_{n_1},\dots,
\underbrace{\mu_{\hat{\ell}},\dots,\mu_{\hat{\ell}}}_{n_{\hat{\ell}}},
\underbrace{\mu_{\hat{\ell}}^{\star},\dots,\mu_{\hat{\ell}}^{\star}}_{n_{\hat{\ell}}},
\underbrace{\mu_{\hat{\ell}+1},\dots,\mu_{\hat{\ell}+1}}_{2n_{\hat{\ell}+1}},
\dots,
\underbrace{\mu_{\hat{t}},\dots,\mu_{\hat{t}}}_{2n_{\hat{t}}},
\end{align}
where $\mu_i\ne \mu_j$ for $i\ne j$, 
$\mu_i\ne \mu_i^{\star}$ for $i=1,\dots,n_{\hat{\ell}}$,
$\mu_i=\mu_i^{\star}$ for $i=n_{\hat{\ell}}+1,\dots,n_{\hat{t}}$.
Now define $\zeta_n(\widetilde{S}S^{-1})=(n_1,n_2,\dots, n_{\hat{t}})$, 
then  $\zeta_n(\widetilde{S}S^{-1})$ is a partition of $n$, i.e., $\sum_{i=1}^{\hat{t}}n_i=n$,
and $\hat{t}$ is the cardinality of $\zeta_n(\widetilde{S}S^{-1})$,
denoted by card$(\zeta_n(\widetilde{S}S^{-1}))$.

\begin{Theorem}\label{thm3.2}
Given $(X,J)\in\mathbb{C}^{n\times 2n}\times \mathbb{C}^{2n\times 2n}$ with 
$\bsmat X\\ -XJ^{-1}\esmat$ nonsingular and $J$ of the PJCF
\footnote{Using similarity transformation, we can show that the assumption that $J$ is of the PJCF can be removed.}.
Assume the geometry multiplicity of each distinct eigenvalue is one.
If there are two nonsingular matrices  $S, \widetilde{S}\in\mathbb{S}_{(X,J)}$ such that
$\zeta_n(\widetilde{S}S^{-1})=(n_1,n_2,\dots, n_{\hat{t}})$, then
for any nonsingular  $\widehat{S}\in\mathbb{S}_{(X,J)}$, 
there exists a nonsingular $K$ such that 
\[
K^{\star}\widehat{Q}(\lambda)K=\diag(Q_1(\lambda),\dots,Q_{\hat{t}}(\lambda)),
\]
where $\widehat{Q}(\lambda)$ is a regular $\star$-(anti)-palindromic system  of the form \eqref{pal}
with coefficient matrices given by \eqref{a1a0} in terms of $X$, $J$ and $\widehat{S}$;
for $i=1,\dots, \hat{t}$, $Q_{i}(\lambda)$ is a regular $\star$-(anti)-palindromic system  of order $n_i$.
Furthermore, the choice of $K$ is independent of $\widehat{S}$.
\end{Theorem}

\begin{proof}
Using $\zeta_n(\widetilde{S}S^{-1})=(n_1,n_2,\dots, n_{\hat{t}})$,
we know that  the eigenvalues of $\widetilde{S}^{-1}S$ are \eqref{eigss}.
Following the proof of Lemma~\ref{lem:jbd1}, we have \eqref{sswaw}.
Then it follows that the eigenvalues of $\widetilde{A}_1^{-1}A_1$ are
\begin{align*}
\underbrace{\hat{\mu}_1,\dots,\hat{\mu}_1}_{n_1},
\dots,
\underbrace{\hat{\mu}_{\hat{\ell}},\dots,\hat{\mu}_{\hat{\ell}}}_{n_{\hat{\ell}}},
\underbrace{\mu_{\hat{\ell}+1},\dots,\mu_{\hat{\ell}+1}}_{n_{\hat{\ell}+1}},
\dots,
\underbrace{\mu_{\hat{t}},\dots,\mu_{\hat{t}}}_{n_{\hat{t}}},
\end{align*}
where $\hat{\mu}_i=\mu_i$ or $\mu_i^{\star}$ for $i=1,\dots,\hat{\ell}$.
Thus, there exists a nonsingular matrix $K$ such that 
\begin{align}\label{kaak}
K^{-1} \widetilde{A}_1^{-\star}A_1^{\star} K=\diag(E_{11},\dots,E_{\hat{t}\hat{t}}),
\end{align}
where $\lambda(E_{ii})=\hat{\mu}_i$ for $i=1,\dots,\hat{\ell}$, 
$\lambda(E_{ii})=\mu_i$ for $i=\hat{\ell}+1,\dots,\hat{t}$.
Using the fact that the eigenvalues of $\widetilde{S}^{-1}S$ are \eqref{eigss} again,
we know that there exists a permutation matrix $\Pi$ such that
\begin{align}\label{pssp}
\Pi^{\top} \widetilde{S}S^{-1} \Pi=\diag(D_{11},\dots, D_{\hat{t}\hat{t}}),
\end{align}
where for $i=1,\dots,\hat{t}$, the eigenvalue set of $D_{ii}$ counting multiplicity is
\[
\lambda(D_{ii})=\{\underbrace{\mu_i,\dots,\mu_i}_{n_i}, \underbrace{\mu_i^{\star},\dots,\mu_i^{\star}}_{n_i}\},
\]
and $\lambda(D_{ii})\cap\lambda(D_{jj})=\emptyset$ for $i\ne j$.
Using Lemma~\ref{lem:jbd2}, we know that $\Pi$ can also be chosen to satisfy 
\begin{align}
\Pi^{\top} S\Pi&=\diag(S_{11},\dots,S_{\hat{t}\hat{t}}),\notag\\
\Pi^{\top}\widetilde{S}\Pi&=\diag(\widetilde{S}_{11},\dots,\widetilde{S}_{\hat{t}\hat{t}}),\label{psp}\\
\Pi^\top J \Pi&=\diag(J_{11},\dots, J_{\hat{t}\hat{t}}),\notag
\end{align}
where for $i=1,\dots,\hat{t}$, $S_{ii}$, $\widetilde{S}_{ii}$ and $J_{ii}$ are all of order $2n_i$.

Now we rewrite \eqref{eigss} as
\begin{align*}
W\widetilde{S}S^{-1}=\diag(\widetilde{A}_1^{-\star}A_1^{\star}, \widetilde{A}_1^{-1}A_1) W,
\end{align*}
and the first $n$ rows become
\begin{align*}
X\widetilde{S}S^{-1}=\widetilde{A}_1^{-\star}A_1^{\star}X.
\end{align*}
Substituting \eqref{pssp} and \eqref{kaak} into the above equality, we get
\begin{align}\label{kxpblock}
K^{-1}X\Pi \diag(D_{11},\dots,D_{\hat{t}\hat{t}})=\diag(E_{11},\dots,E_{\hat{t}\hat{t}})K^{-1}X\Pi.
\end{align}
Partition $K^{-1}X\Pi=[X_{ij}]$, where $X_{ij}\in\mathbb{C}^{n_i\times n_j}$.
Then comparing the $(i,j)$ block of \eqref{kxpblock} on both sides, we have
\[
X_{ij}D_{jj}-E_{ii}X_{ij}=0.
\]
Notice that $\lambda(D_{jj})\cap \lambda(E_{ii})=\emptyset$ for $i\ne j$, thus, $X_{ij}=0$.
Consequently, we can rewrite \eqref{kxpblock} as 
\begin{align}\label{kxp}
K^{-1}X\Pi=\diag(X_{11},\dots,X_{\hat{t}\hat{t}}).
\end{align}

For any nonsingular $\widehat{S}\in\mathbb{S}_{(X,J)}$, we have
\begin{align}\label{phatsp}
\Pi^{\top} \widehat{S}\Pi&=\diag(\widehat{S}_{11},\dots,\widehat{S}_{\hat{t}\hat{t}}).
\end{align}
Then for $\widehat{A}_1$, $\widehat{A}_0$ defined by \eqref{a1a0} in terms of $X$, $J$ and $\widehat{S}$,
using \eqref{kxp}, \eqref{phatsp} and $J$ in \eqref{psp}, we have
\begin{align*}
K^{\star}\widehat{A}_1K
&=\epsilon K^{\star}(XJ^{-1}\widehat{S}X^{\star})^{-1}K\\
&=\epsilon K^{\star} (K\diag(X_{jj})\Pi^{\top} \Pi \diag(J_{jj}^{-1}) \Pi^{\top} \Pi \diag(\widehat{S}_{jj}) \Pi^{\top} \Pi \diag(X_{jj}^{\star})K^{\star})^{-1} K\\
&=\epsilon \diag((X_{jj}J_{jj}^{-1}\widehat{S}_{jj}X_{jj}^{\star})^{-1}),
\end{align*}
similarly,
\begin{align*}
K^{\star} \widehat{A}_0 K= - \diag((X_{jj}J_{jj}^{-1}\widehat{S}_{jj}X_{jj}^{\star})^{-1}
(X_{jj}\widehat{S}_{jj}J_{jj}^{-2}X_{jj}^{\star})
(X_{jj}J_{jj}^{-1}\widehat{S}_{jj}X_{jj}^{\star})^{-1}).
\end{align*}
The conclusion follows.
\qquad \end{proof}

Theorem~\ref{thm3.2} shows that if card$(\zeta_n(\widetilde{S}S^{-1}))=\hat{t}$, then the coefficient matrices can be jointly block diagonalized with $\hat{t}$ diagonal blocks. We then want to show how large can $\hat{t}$ be. Define 
\[
\zeta_n^{\text{opt}}=\zeta_n^{\text{opt}}(\mathbb{S}_{(X,J)})=\text{argmax}\{\text{card}(\zeta_n)|\zeta_n=\zeta_n(\widetilde{S}S^{-1}),S, \widetilde{S}\in\mathbb{S}_{(X,J)} \text{ nonsingular}\}.
\]
Similar as in \cite{cai2015matrix}, we can show that in the definition of $\zeta_n^{\text{opt}}$, $S$ can be fixed as some nonsingular $S_0\in\mathbb{S}_{(X,J)}$, that is,
\[
\zeta_n^{\text{opt}}=\zeta_n^{\text{opt}}(\mathbb{S}_{(X,J)})=\text{argmax}\{\text{card}(\zeta_n)|\zeta_n=\zeta_n(\widetilde{S}S_0^{-1}),\widetilde{S}\in\mathbb{S}_{(X,J)} \text{ nonsingular}\}.
\]
The following theorem characterizes the relationship between card$(\zeta_n^{\text{opt}})$ and the dimension of $\mathbb{S}_{(X,J)}$.

\begin{Theorem}\label{thm3.3}
Given $(X,J)\in\mathbb{C}^{n\times 2n}\times \mathbb{C}^{2n\times 2n}$ with 
$\bsmat X\\ -XJ^{-1}\esmat$ nonsingular and $J$ of the PJCF.
Assume the geometry multiplicity of each distinct eigenvalue is one, it then holds that
\begin{align}\label{dim}
\frac{1}{2}\text{dim}(\mathbb{S}_{(X,J)}) \le \text{card}(\zeta_n^{\text{opt}}) \le \text{dim}(\mathbb{S}_{(X,J)}).
\end{align}
\end{Theorem}

\begin{proof}
For the fixed nonsingular $S_0\in\mathbb{S}_{(X,J)}$, assume that the nonsingular $S\in\mathbb{S}_{(X,J)}$ is chosen such that $\zeta_n(SS_0^{-1})=\zeta_n^{\text{opt}}=(n_1,\cdots,n_t)$. Then from the proof of Theorem~\ref{thm3.2} we know that there exist a nonsingular matrix $K$ and a permutation matrix $\Pi$ such that
\[
K^{-1}X\Pi=\diag(X_{11},\dots,X_{tt}),\quad \Pi^\top J \Pi=\diag(J_{11},\dots, J_{tt}),
\]
where $X_{jj}\in\mathbb{C}^{n_j\times 2n_j},J_{jj}\in\mathbb{C}^{2n_j\times 2n_j}$, and for $A_1$ defined as in \eqref{a1a0} using $X,J,S$, 
\[
K^\star A_1 K=\diag (A_{11},\cdots,A_{1t}),
\]
where $A_{1j}\in\mathbb{C}^{n_j\times n_j}$. Similar to $\mathbb{S}_{(X,J)}$ and $\zeta_n^{\text{opt}}(\mathbb{S}_{(X,J)})$, we can define $\mathbb{S}_{(X_{jj},J_{jj})}$ and $\zeta_n^{\text{opt}}(\mathbb{S}_{(X_{jj},J_{jj})})$, and then 
\[
\text{dim}(\mathbb{S}_{(X,J)})=\sum_{j=1}^t\text{dim}(\mathbb{S}_{(X_{jj},J_{jj})}),\quad \text{card}\left(\zeta_n^{\text{opt}}(\mathbb{S}_{(X_{jj},J_{jj})})\right)=1.
\]
Define $W_{jj}=\begin{bmatrix} X_{jj}\\ -X_{jj}J_{jj}^{-1}\end{bmatrix}$ and $S_{jj}^{(0)}=W_{jj}^\star \begin{bmatrix}0 & -\epsilon A_{1j} \\ A_{1j}^{\star} & 0\end{bmatrix} W_{jj}$, then for any $S_{jj}\in\mathbb{S}_{(X_{jj},J_{jj})}$, the eigenvalues of $S_{jj}(S_{jj}^{(0)})^{-1}$ are either the same $\lambda$ satisfying $\lambda=\lambda^\star$ or a pair of distinct $\lambda$ and $\lambda^\star$. In the former case, $S_{jj}=\lambda S_{jj}^{(0)}$ and hence $\text{dim}(\mathbb{S}_{(X_{jj},J_{jj})})=1$. While in the latter case, if $\text{dim}(\mathbb{S}_{(X_{jj},J_{jj})})>2$, there exist $S_{jj}^{(1)},S_{jj}^{(2)}\in\mathbb{S}_{(X_{jj},J_{jj})}$ such that $S_{jj}^{(k)}(k=0,1,2)$ are linearly independent. Assume that the eigenvalues of $S_{jj}^{(k)}(S_{jj}^{(0))^{-1}}(k=1,2)$ are $\alpha_k\pm i\beta_k (\beta_k\ne 0)$. Let 
\[
S_{jj}=\frac{1}{\beta_1}(S_{jj}^{(1)}-\alpha_1S_{jj}^{(0)})+\frac{2}{\beta_2}(S_{jj}^{(2)}-\alpha_2S_{jj}^{(0)}),
\]
then $S_{jj}\in\mathbb{S}_{(X_{jj},J_{jj})}$ and the eigenvalues of $S_{jj}(S_{jj}^{(0))^{-1}}$ are $\pm 3i$ and $\pm i$. Thus card$(S_{jj}(S_{jj}^{(0))^{-1}})$=2, which contradicts $\text{card}\left(\zeta_n^{\text{opt}}(\mathbb{S}_{(X_{jj},J_{jj})})\right)=1$. So in all $\text{dim}(\mathbb{S}_{(X_{jj},J_{jj})})\le 2$ and hence \eqref{dim} holds.
\end{proof}

Several remarks follow in order.
\begin{enumerate}
\item
By Theorem~\ref{thm3.3}, we know that
when $\text{dim}(\mathbb{S}_{(X,\Lambda)})>3$,
it holds  $t=\text{card}(\zeta_n^{\text{opt}})\ge 2$.
Then it follows from Theorem~\ref{thm3.2} that the coefficient matrices of the quadratic palindromic matrix polynomial can be joint block diagonalized, and the resulting block diagonal matrices has exactly $t$ blocks.

\item
From \cite{cai2009quadratic} to \cite{cai2015matrix},
the authors generalize the results on joint block diagonalization 
of the coefficient matrices of self-adjoint matrix polynomial 
from the quadratic case to high order case.
Similarly, we can generalize the theorems in this subsection to high order palindromic matrix polynomial,
which can be used to solve the general joint block diagonalization problem
of a general matrix set (the matrices in the set are not necessarily Hermitian as in \cite{cai2015matrix}).
The detailed discussions and results will be presented in a separate paper.

\end{enumerate}

\subsection{Case $1\le k< 2n$}
In what follows, we try to give a uniform solution to the IEP-QP($k$). 
We make the following assumptions:
the remaining $2n-k$ eigenvalues of the constructed $Q(\lambda)$ do not intersect the eigenvalues of $T_1$,
and $T_1$ is similar to $T_1^{-\star}$.
The assumption that $T_1$ is similar to $T_1^{-\star}$ amounts to require that eigenvalues of $T_1$ occur in pairs $(\lambda,1/\lambda^{\star})$, and the (algebraic, geometric and partial) multiplicities of eigenvalues in each pair are equal. IEP-QP($k$) with these assumptions is referred to as the IEP-QP($k$)-A hereafter.
The following theorem characterizes the solvability and solutions to IEP-QP($k$)-A.

\begin{Theorem}\label{thm:iepk}
The IEP-QP($k$)-A has a regular solution $Q(\lambda)$ if and only if 
there exist a nonsingular $S_1\in\mathbb{S}_{T_1}$ and a  matrix $\Psi\in\mathbb{C}^{n\times (2n-k)}$
such that $X_1T_1^{-1}S_1X_1^{\star}+Y\Psi\widehat{T}_2^{-1} \Omega\Psi^{\star}Y^{\star}$ is nonsingular and
\begin{align}\label{psi}
\Psi \Omega \Psi^{\star}=-{\Delta},
\end{align}
where $Y$ and $\Delta$ are the $Y$ and $\Delta$ factors of $X_1S_1X_1^{\star}$, respectively,
$\Omega\in\mathbb{C}^{(2n-k)\times (2n-k)}$ is nonsingular, 
and $\Omega=\Delta_{n-p,n-q}$ for $\star=*$ with $p$, $q$ the  positive and negative inertia indices of $\sqrt{-\epsilon}S_1$, respectively,
$\Omega=\Delta_{2n-k}$ for $\star=\top$,
$\widehat{T}_2\in\mathbb{C}^{(2n-k)\times (2n-k)}$ satisfies $\widehat{T}_2\Omega \widehat{T}_2^{\star}=\Omega$.
In such case, the coefficient matrices $A_1$, $A_0$ of $Q(\lambda)$ are given by \eqref{a1a0}
in terms of $X$, $T$ and $S$, where
$X=[X_1 \,\, Y\Psi ]$, $T=\diag(T_1, \widehat{T}_2)$ and $S=\diag(S_1, \Omega)$.
\end{Theorem}
%

\begin{proof}
{\em Necessary.} If IEP-QP($k$)-A has a regular solution $Q(\lambda)$,
then there exist $X_2$ and $T_2$ such that $(X,T)=([X_1\,\, X_2], \diag(T_1,T_2))$ forms a stand pair of $Q(\lambda)$.
Define the parameter matrix $S$ as in \eqref{s}.
By the assumption that $\lambda(T_1)\cap\lambda(T_2)=\emptyset$ and $T_1$ is similar to $T_1^{-\star}$,
we know that $S$ is in a block diagonal form $S=\diag(S_1,S_2)$, 
where $S_1\in\mathbb{C}^{k\times k}$, $S_2\in\mathbb{C}^{(2n-k)\times(2n-k)}$.
By Theorem~\ref{thm:sp}, we have $S\in\mathbb{S}_{(X,T)}$.
Then it follows that $S_1\in\mathbb{S}_{T_1}$, $S_2\in\mathbb{S}_{T_2}$ and 
\begin{equation}\label{x1s1x1x2s2x2}
XSX^{\star}=X_1S_1X_1^{\star}+X_2S_2X_2^{\star}=0.
\end{equation}
Let the $\star$-factorization of $X_1S_1X_1^{\star}$ be $X_1S_1X_1^{\star}=Y\Delta Y^{\star}$,
the $\star$-factorization of $S_2$ be $S_2=G\Omega G^{\star}$.
Then there exists a $\Psi\in\mathbb{C}^{n\times (2n-k)}$ such that $X_2$ as $X_2=Y \Psi G^{-1}$.
Let $\widehat{T}_2=G^{-1} T_2 G$, then it follows from $S_2\in\mathbb{S}_{T_2}$ that
\[
\widehat{T}_2\Omega \widehat{T}_2^{\star}
=G^{-1} T_2 G \Omega G^{\star} T_2^{\star} G^{-\star}
=G^{-1} T_2 S_2 T_2^{\star} G^{-\star}
=G^{-1} S_2  G^{-\star}
=\Omega.
\]

According to Theorem~\ref{thm:sp}, we have $\epsilon A_1^{-1}=XT^{-1}SX^{\star}$, which is nonsingular.
Therefore, the matrix
\begin{align*}
XT^{-1}SX^{\star}
&=X_1T_1^{-1}S_1X_1^{\star}+X_2T_2^{-1}S_2X_2^{\star}\\
&=X_1T_1^{-1}S_1X_1^{\star}+ (Y \Psi G^{-1})( G\widehat{T}_2^{-1} G^{-1}) (G\Omega G^{\star})(Y \Psi G^{-1})^{\star}\\
&=X_1T_1^{-1}S_1X_1^{\star}+ Y \Psi \widehat{T}_2^{-1} \Omega \Psi^{\star} Y^{\star}
\end{align*}
is nonsingular.

Now using the expression of $X_2$ and the $\star$ factorizations of $X_1S_2X_1^{\star}$ and $S_2$, 
we have from \eqref{x1s1x1x2s2x2} that
\[
\Psi \Omega \Psi^{\star}=\Psi G^{-1} G\Omega G^{\star} G^{-\star} \Psi^{\star}
= Y^{-1} X_2S_2X_2^{\star} Y^{-\star}=-Y^{-1} X_1S_1X_1^{\star}Y^{-\star}=-\Delta.
\]

The matrix $\Omega$ is nonsingular since $S_2$ is.
That $\Omega=\Delta_{n-p,n-q}$ for $\star=*$, $\Omega=\Delta_{2n-k}$ for $\star=\top$
comes from the fact that the direct sum of $\Omega$ and the $\Delta$ factor of $S_1$ equals to $\widehat{\Delta}$
up to a permutation,
where $\widehat{\Delta}$ is defined in \eqref{Delta2n}.

{\em Sufficiency.}
First, it is easy to see that $S=\diag(S_1,\Omega)$ is nonsingular since both $S_1$ and $\Omega$ are.
Recall the definition of $\star$-factorization, we know that $\Omega^{\star}=-\epsilon \Omega$.
Then it follows that  $S^{\star}=-\epsilon S$.
By calculation, we have 
\begin{align*}
TST^{\star}
=\diag(T_1S_1T_1^{\star},  \widehat{T}_2 \Omega \widehat{T}_2^{\star} )
=\diag(S_1, \Omega)=S,
\end{align*}
and also
\begin{align*}
XSX^{\star}=Y(\Delta + \Psi \Omega  \Psi^{\star}) Y^{\star}=0.
\end{align*}
Therefore, we have $S\in\mathbb{S}_{(X,T)}$.

Noticing that $XT^{-1}SX^{\star}=X_1T_1S_1X_1^{\star}+ Y \Psi \widehat{T}_2^{-1} \Omega \Psi^{\star} Y^{\star}$ is nonsingular,
we can define $A_1$ and $A_0$ as in \eqref{a1a0}.
Using \eqref{s}, we know that $W=\bsmat X\\ -XT^{-1}\esmat$ is nonsingular since $S$ is.
By Theorem~\ref{thm:iep2n}, a regular solution to IEP-QP($k$)-A can be given by \eqref{a1a0} in terms of $X$, $T$ and $S$.
This completes the proof.
\qquad\end{proof}

By Theorem~\ref{thm:iepk}, we can construct a regular solution to  the IEP-QP($k$)-A as in the following Algorithm~\ref{alg1}.

\begin{algorithm}[htb]\label{alg1}
\KwIn{$(X_1,T_1)\in\mathbb{C}^{n\times k}\times \mathbb{C}^{k\times k}$.}

\KwOut{$A_1$ and $A_0$.}

\caption{Solving IEP-QP($k$)-A}

\BlankLine

\nl Find a nonsingular $S_1\in\mathbb{S}_{T_1}$.

\nl If $\star=*$, compute the positive and negative indices of $\sqrt{-\epsilon}S_1$, denoted by $p$ and $q$, respectively.

\nl Compute the $\star$-factorization of $X_1S_1X_1^{\star}$, denoted by $X_1S_1X_1^{\star}=Y\Delta Y^{\star}$.

\nl Set $\Omega=\Delta_{n-p,n-q}$ for $\star=*$, $\Omega=\Delta_{2n-k}$ for $\star=\top$.

\nl Determine a $\widehat{T}_2$ satisfying $\widehat{T}_2 \Omega \widehat{T}_2^{\star}=\Omega$.

\nl Find a $\Psi$ satisfying \eqref{psi}.

\nl If $X_1T_1^{-1}S_1X_1^{\star}+ Y \Psi \widehat{T}_2^{-1} \Omega \Psi^{\star} Y^{\star}$ is nonsingular, 
compute $A_1$ and $A_0$ as in \eqref{a1a0} in terms of $X=[X_1 \, Y\Psi ]$, $T=\diag(T_1, \widehat{T}_2)$ and $S=\diag(S_1, \Omega)$.
\end{algorithm}

Several remarks follow in order.
\begin{enumerate}
\item[1.] In Step 1, if $T_1$ is of the PJCF, by the discussion  in Section~\ref{ss},
the parameter solution of $S_1$ can be easily obtained, 
from which a nonsingular one can be chosen if there exists;
\item[2.] In Step 5, by arranging $\Omega$ as in \eqref{simples}, a $\widehat{T}_2$ can be easily obtained from \eqref{simplel};
\item[3.] In Step 6, for the case $\star=*$, it must hold that 
\[
n_-(\sqrt{-\epsilon}\Omega)\ge n_+(\sqrt{-\epsilon}\Delta),\qquad
n_+(\sqrt{-\epsilon}\Omega)\ge n_-(\sqrt{-\epsilon}\Delta).
\] 
Otherwise, \eqref{psi} has no solution.
Here $n_+(\cdot)$ and $n_-(\cdot)$ denote the positive and negative indices of a Hermitian matrix, respectively. 
We can of course pursue the general solutions to \eqref{psi}, 
which needs detailed discussions and makes our main idea obscure.
Instead, we can find a $\Psi$ satisfying \eqref{psi} as follows.
If $k\le n$, let $B$ be a nonsingular matrix of order $2n-k$. 
We compute the $\star$-factorization $B\Omega B^{\star}=Y_B \Delta_B Y_B^{\star}$. 
Then let $B:=Y_B^{-1} B$, $\Psi$ can be given by 
\begin{equation}\label{psi1}
\Psi=
\left\{\begin{array}{lll}
\bsmat 0& \Theta_1^* \\ \Theta_2^* & 0\esmat B,
& 
\begin{array}{ll}
\hat{p}=n_+(\sqrt{-\epsilon}\Delta_B),  \Theta_1\in\mathbb{C}^{\hat{p}\times p},  \Theta_1^* \Theta_1=I_p,  \\
\hat{q}=n_-(\sqrt{-\epsilon}\Delta_B),   \Theta_2\in\mathbb{C}^{\hat{q}\times q},  \Theta_2^* \Theta_2=I_q,
\end{array}
&\mbox{if } \star=*, \epsilon=\pm 1; \\
\bsmat 0 & \Theta^{\top}\\ \Theta^{\top} & 0\esmat B , 
&\Theta \in\mathbb{C}^{t/2\times (n-k/2)},\, \Theta^\top \Theta=I,
&\mbox{if } \star=\top, \epsilon=1;
\\
\imath\,\Theta^\top B, 
& \Theta \in\mathbb{C}^{t\times (2n-k)},\, \Theta^\top \Theta=I,
&\mbox{if } \star=\top, \epsilon=-1.
\end{array}\right.
\end{equation}

If $k>n$, let $B\in\mathbb{C}^{n\times (2n-k)}$ be of full column rank.
We compute the $\star$-factorization $B\Omega B^{\star}=Y_B \Delta_B Y_B^{\star}$.
Then let $B:=Y_B^{\dagger} B$, $\Psi$ can also be given by \eqref{psi1}.
\end{enumerate}
%

%
%
%

\section{Solving MUP-QP}\label{eep}
The model updating problem with no-spillover of the quadratic $\star$-(anti)-palindromic system (MUP-QP) can be phrased as follows.

\medskip
{\em\textbf{MUP-QP}: Given a regular $\star$-(anti)-palindromic  system $Q(\lambda)$, 
and some of its eigenpairs $\{(\lambda_j,x_j)\}_{j=1}^k$, 
update $Q(\lambda)$ to another regular $\star$-(anti)-palindromic  system 
$\widetilde{Q}(\lambda)=\lambda^2 \widetilde{A}_1^{\star}+\lambda \widetilde{A}_0+\epsilon \widetilde{A}_1$,
such that $\{\lambda_j\}_{j=1}^k$ are replaced by $\{\tilde{\lambda}_j\}_{j=1}^k$,
meanwhile the remaining $2n-k$ eigenpairs are kept unchanged.}

\medskip
Let $(X_1,T_1)\in\mathbb{C}^{n\times k}\times \mathbb{C}^{k\times k}$, 
$(X_2,T_2)\in\mathbb{C}^{n\times (2n-k)}\times \mathbb{C}^{(2n-k)\times (2n-k)}$ 
be the invariant standard pairs \cite{Betcke2011perturbation} of $Q(\lambda)$ associated with 
$\{(\lambda_j,x_j)\}_{j=1}^k$ and  the remaining $2n-k$ eigenpairs, respectively.
Let $\widetilde{T}_1=\diag(\tilde{\lambda}_1, \tilde{\lambda}_2,\dots, \tilde{\lambda}_k)$,
$\widetilde{X}_1=[\tilde{x}_1 \,\tilde{x}_2 \, \dots \tilde{x}_k]\in\mathbb{C}^{n\times k}$,
where for $j=1,2,\dots, k$, $\tilde{x}_j$ is the eigenvector of $\widetilde{Q}(\lambda)$ corresponding to $\tilde{\lambda}_j$.
Let $X=[X_1 \, X_2]$, $\widetilde{X}=[\widetilde{X}_1 \, X_2]$,
$T=\diag(T_1,T_2)$, $\widetilde{T}=\diag(\widetilde{T}_1,T_2)$.
Then for the original system $Q(\lambda)$, we know that $(X,T)$ is a standard pair of $Q(\lambda)$.
The MUP-QP amounts to find a regular $\star$-(anti)-palindromic  system $\widetilde{Q}(\lambda)$ such that
$(\widetilde{X},\widetilde{T})$ is a standard pair of $\widetilde{Q}(\lambda)$.
%
%


In current literature solving the MUP for symmetric systems, 
the number of unwanted eigenvalues $k$ is usually restricted to be no more than than $n$, 
and the range space spanned by the columns of $\widetilde{X}_1$ is assumed to be in that of $X_1$. 
In what follows we will not make such assumptions, which makes our solutions more general.

\begin{Theorem}\label{thm:eep}
With the notations above, assume that $T_1$ is similar to $T_1^{-\star}$ and $\lambda(T_1)\cap\lambda(T_2)=\emptyset$.
Let 
\begin{align}\label{sss111}
S_1=(\epsilon X_1^{\star}A_1X_1T_1^{-1}-T_1^{-\star}X_1^{\star}A_1^{\star}X_1)^{-1}.
\end{align}
If there exist a nonsingular $\widetilde{S}_1\in\mathbb{S}_{\widetilde{T}_1}$ and a solution $\widetilde{X}_1$ to 
\begin{equation}\label{xsxxsx}
\widetilde{X}_1\widetilde{S}_1\widetilde{X}_1^{\star}=X_1S_1X_1^{\star}
\end{equation}
satisfying that $\Xi:=I_{\ell}+\epsilon Z_2^{\star}A_1Z_1$ is nonsingular,
where $Z_1, Z_2\in\mathbb{C}^{n\times \ell}$ are determined by the rank factorization
$\widetilde{X}_1\widetilde{T}_1^{-1}\widetilde{S}_1\widetilde{X}_1^{\star}  -  X_1T_1^{-1}S_1X_1^{\star}=Z_1Z_2^{\star}$,
then a regular solution to the MUP-QP can be given as
\begin{subequations}\label{a1a0eep}
\begin{align}
\widetilde{A}_1&=A_1-\epsilon A_1 Z_1\Xi^{-1}Z_2^{\star}A_1,\label{a1tilde}\\
\widetilde{A}_0&=(I-\epsilon A_1 Z_1 \Xi^{-1}Z_2^{\star}) 
(A_0-A_1 \Upsilon A_1) 
(I-\epsilon Z_1 \Xi^{-1}Z_2^{\star}A_1),\label{a0tilde}
\end{align}
\end{subequations}
where $\Upsilon=\widetilde{X}_1\widetilde{T}_1^{-2}\widetilde{S}_1\widetilde{X}_1^{\star}-X_1T_1^{-2}S_1X_1^{\star}$.
\end{Theorem}

\begin{proof}
For the original system $Q(\lambda)$, let $S$ be the parameter matrix given by
\begin{align}
S=
\left(\bsmat X_1 & X_2\\ -X_1T_1^{-1} & -X_2T_2^{-1}\esmat^{\star} 
\bsmat 0 & -\epsilon A_1\\ A_1^{\star} & 0\esmat
 \bsmat X_1 & X_2\\ -X_1T_1^{-1} & -X_2T_2^{-1}\esmat\right)^{-1}.
\end{align}
Using the assumptions that  $T_1$ is similar to $T_1^{-\star}$ 
and $\lambda(T_1)\cap\lambda(T_2)=\emptyset$,
we know that $S$ is in the block diagonal form $S=\diag(S_1,S_2)$, where $S_1$ is of order $k$ and given by \eqref{sss111}, and
$S_2$ is of order $2n-k$ and given by $S_2=(\epsilon X_2^{\star}A_1X_2T_2^{-1}-T_2^{-\star}X_2^{\star}A_1^{\star}X_2)^{-1}$.
By Theorem~\ref{thm:sp}, it holds that $S\in\mathbb{S}_{(X,T)}$, i.e, 
$S^{\star}=-\epsilon S$, $S=TST^{\star}$ and $XSX^{\star}=0$.

Now let $\widetilde{S}=\diag(\widetilde{S}_1,S_2)$.
Then $\widetilde{S}$ is nonsingular since both $\widetilde{S}_1$ and $S_2$ are. 
On one hand,
using $\widetilde{S}_1\in\mathbb{S}_{\widetilde{T}_1}$ and $S_2\in\mathbb{S}_{T_2}$ (since $S\in\mathbb{S}_{T}$),
we have $\widetilde{S}\in\mathbb{S}_{\widetilde{T}}$.
On the other hand, using \eqref{xsxxsx}, we have
\begin{align*}
\widetilde{X}\widetilde{S}\widetilde{X}^{\star}
=\widetilde{X}_1\widetilde{S}_1\widetilde{X}_1^{\star}+X_2S_2X_2^{\star}
=X_1S_1X_1^{\star}+X_2S_2X_2^{\star}=0.
\end{align*}
Therefore, $\widetilde{S}\in\mathbb{S}_{(\widetilde{X},\widetilde{T})}$.

Using \eqref{a1tilde}, direct calculations give
\begin{align*}
\widetilde{A}_1(A_1^{-1}+\epsilon Z_1Z_2^{\star})
&=(A_1-\epsilon A_1 Z_1\Xi^{-1}Z_2^{\star}A_1)(A_1^{-1}+\epsilon Z_1Z_2^{\star})\\
&=I+\epsilon A_1Z_1Z_2^{\star}-\epsilon A_1Z_1\Xi^{-1}Z_2^{\star}-
A_1 Z_1\Xi^{-1}Z_2^{\star}A_1Z_1Z_2^{\star}\\
&=I + \epsilon A_1Z_1 \Xi^{-1}(\Xi -I-\epsilon Z_2^{\star}A_1Z_1 )Z_2^{\star}=I.\footnotemark
\end{align*}
\footnotetext{This is actually a proof of Sherman–-Morrison–-Woodbury formula.}
Then it follows that $\widetilde{A}_1$ is nonsingular and
\begin{align}\label{a1tilde2}
\widetilde{A}_1^{-1}&=A_1^{-1}+\epsilon Z_1Z_2^{\star}\notag\\
&=\epsilon XT^{-1}SX^{\star} 
+\epsilon ( \widetilde{X}_1\widetilde{T}_1^{-1}\widetilde{S}_1\widetilde{X}_1^{\star}  -  X_1T_1^{-1}S_1X_1^{\star})\\
&=\epsilon ( \widetilde{X}_1\widetilde{T}_1^{-1}\widetilde{S}_1\widetilde{X}_1^{\star}
+ X_2T_2^{-1}S_2X_2^{\star})=\epsilon \widetilde{X}\widetilde{T}^{-1}\widetilde{S}\widetilde{X}^{\star}\notag, 
\end{align}
where the second equality uses the expression of $A_1$ in \eqref{a1a0}.

Using the expression of $A_0$ in \eqref{a1a0} and \eqref{a0tilde}, we have
\begin{align*}
\widetilde{A}_1^{-1}\widetilde{A}_0\widetilde{A}_1^{-1}&=A_1^{-1}(A_0-A_1 \Upsilon A_1)A_1^{-1}\\
&=A_1^{-1}A_0A_1^{-1} - 
(\widetilde{X}_1\widetilde{T}_1^{-2}\widetilde{S}_1\widetilde{X}_1^{\star}-X_1T_1^{-2}S_1X_1^{\star} )\\
&=-XT^{-2}SX^{\star}-\widetilde{X}_1\widetilde{T}_1^{-2}\widetilde{S}_1\widetilde{X}_1^{\star}+X_1T_1^{-2}S_1X_1^{\star}=- \widetilde{X}\widetilde{T}^{-2}\widetilde{S}\widetilde{X}^{\star},
\end{align*}
which can be rewritten as
\begin{align}\label{a0tilde2}
\tilde{A}_0=-\widetilde{A}_1\widetilde{X}\widetilde{T}^{-2}\widetilde{S}\widetilde{X}^{\star}\widetilde{A}_1.
\end{align}

Notice that $\widetilde{A}_1$ in \eqref{a1tilde2} and $\widetilde{A}_0$ in \eqref{a0tilde2} 
 are actually $A_1$ and $A_0$ in \eqref{a1a0} in terms of 
$X=\widetilde{X}$, $T=\widetilde{T}$ and $S=\widetilde{S}$.
According to Theorem~\ref{thm:iep2n}, $(\widetilde{X},\widetilde{T})$ is a standard pair of $\widetilde{Q}(\lambda)$.
This completes the proof.
\qquad \end{proof}

By Theorem~\ref{thm:eep}, we can construct a solution to MUP-QP as in the following Algorithm~\ref{alg2}.
%
%
%
%
%
%

\begin{algorithm}[htb]\label{alg2}
\KwIn{The original system $Q(\lambda)=\lambda^2A_1^\star+\lambda A_0+A_1$, some of its eigenpairs in $(X_1,T_1)\in\mathbb{C}^{n\times k}\times \mathbb{C}^{k\times k}$ and $\widetilde{T}_1=\diag(\tilde{\lambda}_1, \tilde{\lambda}_2,\dots, \tilde{\lambda}_k)$.}

\KwOut{$\widetilde{A}_1$ and $\widetilde{A}_0$.}

\caption{Solving the MUP-QP}

\BlankLine

\nl Compute $S_1$ as in \eqref{sss111} and the $\star$-factorization 
$X_1S_1X_1^{\star}=Y\Delta Y^{\star}$.

\nl Construct a nonsingular $\widetilde{S}_1\in\mathbb{S}_{\widetilde{T}_1}$.

\nl Compute the $\star$-factorization  $\widetilde{S}_1=\widetilde{Y}\widetilde{\Delta}\widetilde{Y}^{\star}$.

\nl Solve $\Psi \widetilde{\Delta} \Psi^{\star}=\Delta$ for $\Psi$, 
which can be obtained similar as \eqref{psi}.

\nl Compute $\widetilde{X}_1=Y\Psi\widetilde{Y}^{-1}$.

\nl Compute the rank factorization $\widetilde{X}_1\widetilde{T}_1^{-1}\widetilde{S}_1\widetilde{X}_1^{\star}  -  X_1T_1^{-1}S_1X_1^{\star}=Z_1Z_2^{\star}$.

\nl Compute $\Xi$ and $\Upsilon$ as in Theorem~\ref{thm:eep}.

\nl If $\Xi$ is nonsingular, compute $\widetilde{A}_1$ and $\widetilde{A}_0$ as in \eqref{a1a0eep}. 
\end{algorithm}

In some applications, the eigenvector matrix $\widetilde{X}_1$ is prescribed. In this case, if there exists a nonsingular $\widetilde{S}_1\in\mathbb{S}_{\widetilde{T}_1}$ such that it holds  \eqref{xsxxsx}
and $\Xi:=I_{\ell}+\epsilon Z_2^{\star}A_1Z_1$ is nonsingular,
where $Z_1, Z_2\in\mathbb{C}^{n\times \ell}$ are determined by the rank factorization
$\widetilde{X}_1\widetilde{T}_1^{-1}\widetilde{S}_1\widetilde{X}_1^{\star}  -  X_1T_1^{-1}S_1X_1^{\star}=Z_1Z_2^{\star}$,
then a regular solution can be given by $\widetilde{Q}(\lambda)$ with $A_1$, $A_0$ given by \eqref{a1a0eep}.

We need to solve a nonsingular $S_1\in\mathbb{S}_{\widetilde{T}_1}$ satisfying \eqref{xsxxsx}.
If $\widetilde{T}_1$ is of the PJCF, by the structure of the parameter matrix,
we can get the parameter expression of $\widetilde{S}_1$.
Noticing that \eqref{xsxxsx} is a linear system of equations,
we can solve the parameters of $\widetilde{S}_1$ with ease.
However,  whether $\widetilde{S}_1$ is nonsingular or not, determined by those parameters, depends on.
In \cite{cai2009quadratic}, under the assumption that all eigenvalues are simple,
a necessary and sufficient condition is given for the existence of nonsingular $S$,
and a numerical method is proposed to find a nonsingular one.
Here, assuming that all eigenvalues of $\widetilde{T}_1$ is simple
and using the parametric expression of $\widetilde{S}_1$,
we can follow the approach in \cite{cai2009quadratic} to find a nonsingular $\widetilde{S}_1\in\mathbb{S}_{\widetilde{T}_1}$
when there exists.

It is also worth mentioning here that in the above theorem
the number of updated eigenpairs $k$ is not required to be no more than $n$,
and the range space spanned by  the column vectors of $\widetilde{X}_1$ is not necessarily in that of $X_1$. 

\section{Numerical Examples}\label{example}
In this section, we will present some examples to illustrate the performance of Algorithms~\ref{alg1} and \ref{alg2} for solving the IEP-QP$(k)$-A and MUP-QP, respectively, on four different types of palindromic systems as shown in \eqref{4pal}.

\textbf{Example 1.}\quad The following are some examples illustrating the performance of Algorithm~\ref{alg1} for solving the IEP-QP$(k)$-A. Setting
\[
X_1=\begin{bmatrix}1 &i &0 &0\\2 &2i &1 &0\\1 &1 &i &i\\1 &-1 &1 &-1\end{bmatrix},\qquad T_1=\diag\left(1+i,\frac{1}{1+i},2+3i,\frac{1}{2+3i}\right),
\]
Algorithm~\ref{alg1} computes a $\top$-palindromic system $Q(\lambda)=\lambda^2 A_1^{\top} +\lambda A_0 + A_1$ with
\begin{align*}
A_1&=\begin{bmatrix}4.3000 - 3.1500i  &-1.4500 + 1.3500i  &-1.1750 - 0.4750i   &0.2750 - 0.5750i\\
  -1.4500 + 2.3500i   &0.5000 - 1.0000i   &0.6500 + 0.3000i  &-0.2000 + 0.3500i\\
  -1.4250 - 1.7250i   &0.6500 + 0.8000i  &-0.1250 + 0.2500i  &-0.2500 - 0.1250i\\
  -0.4750 - 0.8250i   &0.3000 + 0.3500i  &-0.2500 - 0.1250i   &0.1250 - 0.2500i\end{bmatrix},\\
A_0&=\begin{bmatrix} -8.6000 + 6.3000i   &2.9000 - 3.7000i   &2.6000 + 2.2000i   &0.2000 + 1.4000i\\
   2.9000 - 3.7000i  &-1.0000 + 2.0000i  &-1.3000 - 1.1000i  &-0.1000 - 0.7000i\\
   2.6000 + 2.2000i  &-1.3000 - 1.1000i   &0.2500 - 0.5000i   &0.5000 + 0.2500i\\
   0.2000 + 1.4000i  &-0.1000 - 0.7000i   &0.5000 + 0.2500i  &-0.2500 + 0.5000i\end{bmatrix},
\end{align*} 
which satisfies
\[
\|A_0-A_0^\top\|_2=6.7641\times 10^{-15},\quad \|A_1^\top X_1T_1^2+A_0X_1T_1+A_1X_1\|_2=1.6859\times 10^{-14},
\]
and also a $\top$-anti-palindromic system $Q(\lambda)=\lambda^2 A_1^{\top} +\lambda A_0 - A_1$ with
\begin{align*}
A_1&=\begin{bmatrix}-1.5000 - 0.6250i   &0.9750 + 0.5750i  &-0.3375 + 0.2625i  &-0.3625 - 0.0375i\\
   0.2750 - 0.3250i  &-0.2500 + 0.0000i   &0.0750 - 0.1000i   &0.1500 - 0.0750i\\
   0.3375 + 0.7375i  &-0.0750 - 0.4000i   &0.0625 - 0.0000i   &0.0000 + 0.0625i\\
   0.3625 + 0.0375i  &-0.1500 + 0.0750i  &-0.0000 + 0.0625i  &-0.0625 + 0.0000i\end{bmatrix},\\
A_0&=\begin{bmatrix} 0.0000 + 0.0000i   &0.7000 + 0.9000i  &-0.6750 - 0.4750i  &-0.7250 - 0.0750i\\
  -0.7000 - 0.9000i   &0.0000 + 0.0000i  &0.1500 + 0.3000i   &0.3000 - 0.1500i\\
   0.6750 + 0.4750i  &-0.1500 - 0.3000i   &0.0000 + 0.0000i   &0.0000 - 0.0000i\\
   0.7250 + 0.0750i  &-0.3000 + 0.1500i   &0.0000 - 0.0000i  &0.0000 + 0.0000i\end{bmatrix},
\end{align*}
which satisfies
\[
\|A_0+A_0^\top\|_2=7.3444\times 10^{-15},\quad \|A_1^\top X_1T_1^2+A_0X_1T_1-A_1X_1\|_2=6.5052\times 10^{-15}.
\] 
Setting
\[
X_1=\begin{bmatrix}1 &i &0 &0\\2 &2i &1 &0\\1 &1 &i &i\\1 &-1 &1 &-1\end{bmatrix},\qquad T_1=\diag\left(1+i,\frac{1}{1-i},2+3i,\frac{1}{2-3i}\right),
\]
Algorithm~\ref{alg1} computes a $\ast$-palindromic system $Q(\lambda)=\lambda^2 A_1^* +\lambda A_0 + A_1$ with
\begin{align*}
A_1&=\begin{bmatrix}5.1000 + 0.7500i  &-2.6500 - 0.0500i   &0.6750 - 1.2250i   &1.0250 + 0.0750i\\
  -2.1500 - 1.4500i   &1.0000 + 0.5000i  &-0.4000 + 0.5500i  &-0.4500 - 0.1000i\\
  -0.3250 + 2.2250i   &0.1000 - 1.0500i   &0.2500 + 0.1250i  &-0.1250 + 0.2500i\\
   0.0250 + 0.9250i   &0.0500 - 0.4000i   &0.1250 - 0.2500i   &0.2500 + 0.1250i\\\end{bmatrix},\\
A_0&=\begin{bmatrix} -10.2000 - 0.0000i   &4.8000 - 1.4000i  &-0.3500 + 3.4500i  &-1.0500 + 0.8500i\\
   4.8000 + 1.4000i  &-2.0000 + 0.0000i   &0.3000 - 1.6000i   &0.4000 - 0.3000i\\
  -0.3500 - 3.4500i   &0.3000 + 1.6000i  &-0.5000 - 0.0000i  &-0.0000 - 0.5000i\\
  -1.0500 - 0.8500i   &0.4000 + 0.3000i   &0.0000 + 0.5000i  &-0.5000 + 0.0000i\end{bmatrix},
\end{align*} 
which satisfies
\[
\|A_0-A_0^*\|_2=2.0551\times 10^{-14},\quad \|A_1^* X_1T_1^2+A_0X_1T_1+A_1X_1\|_2=4.3688\times 10^{-14},
\]
and also a $\ast$-anti-palindromic system $Q(\lambda)=\lambda^2 A_1^* +\lambda A_0 - A_1$ with
\begin{align*}
A_1&=\begin{bmatrix} -2.0000 + 2.0500i   &0.1500 - 0.4500i   &0.1750 - 0.2250i   &0.0250 - 0.4250i\\
   0.8500 + 0.5500i   &0.0000 - 0.5000i   &0.1000 + 0.0500i   &0.0500 + 0.4000i\\
   0.0750 - 1.4750i  &-0.1000 + 0.5500i   &0.0000 - 0.1250i   &0.1250 + 0.0000i\\
   0.7250 - 0.6750i  &-0.5500 + 0.4000i  &-0.1250 - 0.0000i   &0.0000 - 0.1250i\end{bmatrix},\\
A_0&=\begin{bmatrix}  0.0000 + 4.1000i  &-0.7000 + 0.1000i   &0.1000 - 1.7000i  &-0.7000 - 1.1000i\\
   0.7000 + 0.1000i   &0.0000 - 1.0000i   &0.2000 + 0.6000i   &0.6000 + 0.8000i\\
  -0.1000 - 1.7000i  &-0.2000 + 0.6000i   &0.0000 - 0.2500i   &0.2500 - 0.0000i\\
   0.7000 - 1.1000i  &-0.6000 + 0.8000i  &-0.2500 + 0.0000i   &0.0000 - 0.2500i\end{bmatrix},
\end{align*}
which satisfies
\[
\|A_0+A_0^*\|_2=7.5570\times 10^{-15},\quad \|A_1^* X_1T_1^2+A_0X_1T_1-A_1X_1\|_2=7.9397\times 10^{-15},
\]

\medskip
\textbf{Example 2.}\quad The following are some examples illustrating the performance of Algorithm~\ref{alg2} for solving the MUP-QP.

Let
\[
A_1=\begin{bmatrix}2 &1+2i &1-2i\\ 1  &-1+i &1+i\\ 1-2i &1+i &1\end{bmatrix},\quad
A_0=\begin{bmatrix}4 &-3+i &5\\ -3+i &1 &-1\\5 &-1 &-1\end{bmatrix}.
\]
Algorithm~\ref{alg2} updates the original $\top$-palindromic system $Q(\lambda)=\lambda^2 A_1^\top +\lambda A_0 + A_1$ to a new $\top$-palindromic system $\widetilde{Q}(\lambda)=\lambda^2 \widetilde{A}_1^\top+\lambda \widetilde{A}_0+\widetilde{A}_1$, such that two eigenvalues $\lambda_1=-4.0685 +10.3032i$ and $\lambda_2=-0.0332 - 0.0840i=1/{\lambda}_1$ are replaced by $\tilde{\lambda}_1=-6+9i, \tilde{\lambda}_2=1/(-6+9i)$, while the remaining eigenpairs are kept unchanged. The updated $\top$-palindromic system is given by
\begin{align*}
\widetilde{A}_1&=\begin{bmatrix}6.3172 - 1.0938i   &1.1789 + 0.7591i   &1.2710 - 1.0174i\\
  -2.1052 + 3.9269i  &-0.1964 + 0.4151i  &-0.0471 + 0.7540i\\
   2.2610 - 5.5507i   &1.2274 + 1.7493i   &1.6496 - 0.1509i\end{bmatrix},\\
\widetilde{A}_0&=\begin{bmatrix}-10.7835 -10.6064i   &6.0838 + 8.1298i   &1.6969 - 6.9283i\\
   6.0838 + 8.1298i  &-2.8344 - 3.3986i   &0.1363 + 3.6357i\\
   1.6969 - 6.9283i   &0.1363 + 3.6357i  &-1.2954 - 2.9336i\end{bmatrix},
\end{align*}
and satisfies
\begin{gather*}
\|\widetilde{A}_0-\widetilde{A}_0^\top\|_2=3.8603\times 10^{-13},\quad \|\widetilde{A}_1^\top \widetilde{X}_1\widetilde{T}_1^2+\widetilde{A}_0\widetilde{X}_1\widetilde{T}_1+\widetilde{A}_1\widetilde{X}_1\|_2=7.8410\times 10^{-12},\\
\|\widetilde{A}_1^\top X_2T_2^2+\widetilde{A}_0X_2T_2+\widetilde{A}_1X_2\|_2=6.6798\times 10^{-13},
\end{gather*}
where $T_1=\diag(\tilde{\lambda}_1, \tilde{\lambda}_2)$, and $\widetilde{X}_1$ are corresponding eigenvectors, while $(X_2,T_2)$ are the remaining eigenpairs of the original system $Q(\lambda)$ to be kept unchanged.

Let
\[
A_1=\begin{bmatrix}2 &1 &1\\ 1  &-1 &1\\ 1 &1 &1\end{bmatrix},\quad
A_0=\begin{bmatrix}0 &-3 &5\\3 &0 &-1\\-5 &1 &0\end{bmatrix}.
\]
Algorithm~\ref{alg2} updates the original $\top$-anti-palindromic system $Q(\lambda)=\lambda^2 A_1^\top +\lambda A_0 - A_1$ to a new $\top$-anti-palindromic system $\widetilde{Q}(\lambda)=\lambda^2 \widetilde{A}_1^\top+\lambda \widetilde{A}_0-\widetilde{A}_1$, such that two eigenvalues $\lambda_1=4.2361$ and $\lambda_2=0.2361=1/{\lambda}_1$ are replaced by $\tilde{\lambda}_1=4, \tilde{\lambda}_2=1/4$, while the remaining eigenpairs are kept unchanged. The updated $\top$-anti-palindromic system is given by
\begin{align*}
\widetilde{A}_1=\begin{bmatrix}-2.3558    &2.8351   &-0.6066\\
    9.3284   &-4.9928    &4.8786\\
  -11.4287    &7.0430   &-4.9287\\
\end{bmatrix},\quad
\widetilde{A}_0=\begin{bmatrix} 0.0000  &-0.0961    &0.1602\\
    0.0961   &0.0000   &-0.0320\\
   -0.1602    &0.0320   &0.0000\end{bmatrix},
\end{align*}
and satisfies
\begin{gather*}
\|\widetilde{A}_0+\widetilde{A}_0^\top\|_2=1.3555\times 10^{-12},\quad \|\widetilde{A}_1^\top \widetilde{X}_1\widetilde{T}_1^2+\widetilde{A}_0\widetilde{X}_1\widetilde{T}_1-\widetilde{A}_1\widetilde{X}_1\|_2=4.5452\times 10^{-13},\\
\|\widetilde{A}_1^\top X_2T_2^2+\widetilde{A}_0X_2T_2-\widetilde{A}_1X_2\|_2=1.2552\times 10^{-12},
\end{gather*}

Let
\[
A_1=\begin{bmatrix}2-5i &1+2i &1-2i\\ 1+2i  &-1+i &1+i\\ 1-2i &1+i &1+3i\end{bmatrix},\quad
A_0=\begin{bmatrix}4 &-3 &5\\ -3 &1 &-1\\5 &-1 &-1\end{bmatrix}.
\]
Algorithm~\ref{alg2} updates the original $\ast$-palindromic system $Q(\lambda)=\lambda^2 A_1^* +\lambda A_0 + A_1$ to a new $\ast$-palindromic system $\widetilde{Q}(\lambda)=\lambda^2 \widetilde{A}_1^*+\lambda \widetilde{A}_0+\widetilde{A}_1$, such that two eigenvalues $\lambda_1=0.8745 + 0.6115i$ and $\lambda_2=0.7680 + 0.5371i=1/\bar{\lambda}_1$ are replaced by $\tilde{\lambda}_1=1+i, \tilde{\lambda}_2=1/(1-i)$, while the remaining eigenpairs are kept unchanged. The updated $\ast$-palindromic system is given by
\begin{align*}
\widetilde{A}_1&=\begin{bmatrix}6.7472 - 9.3739i   &2.5735 + 4.0973i   &0.7229 - 5.6862i\\
   3.0311 + 3.2620i  &-1.5796 + 1.6936i   &2.5406 + 0.6804i\\
   3.1161 - 2.9941i   &1.5673 + 2.0398i   &0.4722 + 2.4112i\end{bmatrix},\\
\widetilde{A}_0&=\begin{bmatrix}16.4611 - 0.0000i  &-2.3817 + 1.4735i   &9.3482 - 4.0036i\\
  -2.3817 - 1.4735i   &2.6838 - 0.0000i  &-1.0836 - 0.5791i\\
   9.3482 + 4.0036i  &-1.0836 + 0.5791i  &-0.6201 - 0.0000i\end{bmatrix},
\end{align*}
and satisfies
\begin{gather*}
\|\widetilde{A}_0-\widetilde{A}_0^*\|_2=6.3131\times 10^{-13},\quad \|\widetilde{A}_1^* \widetilde{X}_1\widetilde{T}_1^2+\widetilde{A}_0\widetilde{X}_1\widetilde{T}_1+\widetilde{A}_1\widetilde{X}_1\|_2=1.3634\times 10^{-13},\\
\|\widetilde{A}_1^* X_2T_2^2+\widetilde{A}_0X_2T_2+\widetilde{A}_1X_2\|_2=7.9817\times 10^{-13}.
\end{gather*}

Let
\[
A_1=\begin{bmatrix}2-5i &1+2i &1-2i\\ 1+2i  &-1+i &1+i\\ 1-2i &1+i &1+3i\end{bmatrix},\quad
A_0=\begin{bmatrix}0 &-3 &5\\ -3 &0 &-1\\5 &-1 &0\end{bmatrix}.
\]
Algorithm~\ref{alg2} updates the original $\ast$-anti-palindromic system $Q(\lambda)=\lambda^2 A_1^* +\lambda A_0 - A_1$ to a new $\ast$-palindromic system $\widetilde{Q}(\lambda)=\lambda^2 \widetilde{A}_1^*+\lambda \widetilde{A}_0-\widetilde{A}_1$, such that two eigenvalues $\lambda_1=0.8195 - 2.4199i$ and $\lambda_2=0.1255 - 0.3707i=1/\bar{\lambda}_1$ are replaced by $\tilde{\lambda}_1=1-2.5i, \tilde{\lambda}_2=1/(1+2.5i)$, while the remaining eigenpairs are kept unchanged. The updated $\ast$-anti-palindromic system is given by
\begin{align*}
\widetilde{A}_1&=\begin{bmatrix} 2.1491 - 6.5859i  &-0.6259 + 4.1774i   &0.5897 - 5.0508i\\
   0.6137 + 0.3797i  &-0.5406 + 1.0722i   &0.9749 + 0.1034i\\
  -0.2765 + 0.4966i   &0.7322 + 0.2886i   &0.0448 + 4.3964i\end{bmatrix},\\
\widetilde{A}_0&=\begin{bmatrix} 0.0000 + 0.3634i   &1.5570 + 2.9811i  &-0.0006 + 0.1029i\\
  -1.5570 + 2.9811i   &0.0000 - 1.9361i  &-1.9504 + 1.2902i\\
   0.0006 + 0.1029i   &1.9504 + 1.2902i   &0.0000 + 0.5061i\end{bmatrix},
\end{align*}
and satisfies
\begin{gather*}
\|\widetilde{A}_0+\widetilde{A}_0^*\|_2=5.8613\times 10^{-14},\quad \|\widetilde{A}_1^* \widetilde{X}_1\widetilde{T}_1^2+\widetilde{A}_0\widetilde{X}_1\widetilde{T}_1-\widetilde{A}_1\widetilde{X}_1\|_2=4.6765\times 10^{-14},\\
\|\widetilde{A}_1^* X_2T_2^2+\widetilde{A}_0X_2T_2-\widetilde{A}_1X_2\|_2=8.4643\times 10^{-14}.
\end{gather*}

\section{Conclusion}\label{conclusion}
In this paper we consider some inverse eigenvalue problems of quadratic palindromic systems, namely, 
the inverse eigenvalue problem with $k$ prescribed eigenpairs (IEP-QP($k$)) and the model updating problem with no-spillover (MUP-QP). Solutions to the IEP-QP($k$) are given uniformly without distinguishing $k\le n$ and $k>n$. And for the IEP-QP($2n$), we show under what condition the coefficient matrices of the solutions 
can be jointly block diagonalized. We also give parametric solutions to the MUP-QP, without assuming that the number of the unwanted eigenvalues/eigenpairs is no more than $n$ and that the space spanned by the column vectors of $\widetilde{X}_1$ is in that of $X_1$. 

\bibliographystyle{abbrv} 
%

%

\end{document}